\documentclass[leqno,11pt,letterpaper, english]{amsart}
\usepackage[usenames,dvipsnames]{color}
\usepackage[colorlinks=true,linkcolor=Red,citecolor=Green]{hyperref}
\usepackage{amsmath,amssymb,amsthm,graphicx,url,bbm}
\usepackage{epstopdf}
\usepackage{color}
\usepackage{enumerate}
\usepackage{dsfont}
\usepackage{mathtools}
\usepackage[T1]{fontenc}
\usepackage{tikz}
\usepackage[
	backend=biber,
	citestyle=alphabetic,]{biblatex}
\newcommand{\xqedhere}[1]{%
    \rlap{%
         \hbox to#1{%
           \hfil
           \llap{%
               \ensuremath{\square}
           }%
       }%
   }%
}

\def\pasdegrille{\let\grille = \pasgrille}

\def\aat#1#2#3{
\divide \dimen1 by 48 \dimen3=\dimen1 \multiply \dimen1 by #1
\advance \dimen1 by -\dimen3 \divide \dimen1 by 101 \multiply
\dimen1 by 100 \divide \dimen2 by \count11 \multiply \dimen2 by #2
\setbox0=\hbox{#3}\ht0=0pt\dp0=0pt
  \rlap{\kern\dimen1 \vbox to0pt{\kern-\dimen2\box0\vss}}\dimen1= \wd1
\dimen2=\ht1}
\def\pasgrille{
\count12= \dimen1 \divide \count12 by 50 \divide \dimen2 by
\count12 \count11 =\dimen2 \ \divide \dimen1 by 48
\setlength{\unitlength}{\dimen1} \smash{\rlap{\ }} \dimen1= \wd1
\dimen2=\ht1 }
\def\grille{
\count12= \dimen1 \divide \count12 by 50 \divide \dimen2 by
\count12 \count11 =\dimen2 \ \divide \dimen1 by 48
\setlength{\unitlength}{\dimen1}
\smash{\rlap{\graphpaper[1](0,0)(50, \count11)}} \dimen1= \wd1
\dimen2=\ht1 }
\pasdegrille
\newcommand{\be}{\begin{equation}}
\newcommand{\ee}{\end{equation}}
\newcommand{\ra}{\rangle}
\newcommand{\la}{\langle}

\newcommand{\D}{{\mathcal D}}

\newcommand{\HH}{{\mathcal H}}

\newcommand{\cM}{{\mathcal M}}

\newcommand{\Oo}{{\mathcal O}}

\newcommand{\UU}{{\mathcal U}}

\newcommand{\CC}{{\mathcal C}}
\newcommand{\Aa}{{\mathcal A}}
\newcommand{\Pp}{{\mathcal P}}
\newcommand{\Ll}{{\mathcal L}}

\newcommand{\Gg}{{\mathcal G}}
\newcommand{\leeq}{\lesssim}

\newcommand{\tr}{\operatorname{tr}}

\theoremstyle{plain}

\newtheorem{thm}{Theorem}

\newtheorem{prop}{Proposition}[section]
\newtheorem{cor}[prop]{Corollary}
\newtheorem{lem}[prop]{Lemma}

\theoremstyle{definition}

\newtheorem{rem}[prop]{Remark}
\newtheorem{defn}[prop]{Definition} 

\numberwithin{equation}{section}

\def\squarebox#1{\hbox to #1{\hfill\vbox to #1{\vfill}}}

\usepackage{amsxtra}

\ifx\pdfoutput\undefined
  \DeclareGraphicsExtensions{.pstex, .eps}
\else
  \ifx\pdfoutput\relax
    \DeclareGraphicsExtensions{.pstex, .eps}
  \else
    \ifnum\pdfoutput>0
      \DeclareGraphicsExtensions{.pdf}
    \else
      \DeclareGraphicsExtensions{.pstex, .eps}
    \fi
  \fi
\fi

\title{2D Torus defocusing NLSE with a potential and random initial data}

\author[N. Lindstrom]{Nicolas Lindstrom}
\address{Universit{\'e} Paris-Saclay, Math{\'e}matiques, UMR 8628 du CNRS, B{\^a}t 307, 91405  Orsay Cedex, France,   and Institut Universitaire de France}
\email{nicolas.lindstrom@universite-paris-saclay.fr}

\usepackage{amssymb}
\usepackage{amsmath, amsthm, amsopn, amsfonts}

\def\11{{\rm 1~\hspace{-1.4ex}l} }
\def\R{\mathbb R}

\def\Z{\mathbb Z}
\def\N{\mathbb N}
\def\E{\mathbb E}
\def\T{\mathbb T}

\begin{document}    

\begin{abstract}
We consider the non-linear cubic defocusing Schrödinger equation on the 2D Torus with a potential $V(x)\in H^{2^+}(\T^2)$. The initial data, $u_0$,is given by a centered gaussian  variable with covariance $(-\Delta+V)^{-1}$,  and thus is almost surely in $H^{0^-}(\T^2)$. We show that it is almost surely well-possed for all time in the weak sense, and that $u-e^{it(\Delta -V)}u_0$ is in $H^{s}(\T)^2$ for some $s>0$, and that the solution is the limit of the solutions to the truncated equation. This result is a generalization of the results first given in \cite{2D_Defocusing}.
 \begin{center} \rule{0.6\textwidth}{.4pt}
\end{center}
\end{abstract}   

\ \vskip -1cm \noindent\hfil\rule{0.9\textwidth}{.4pt}\hfil \vskip 1cm 
 \maketitle
\section{Introduction}
\subsection{Context and Motivation}
Our motivation is to show that the results proved by Bourgain in \cite{2D_Defocusing} on the non linear cubic Schrödinger equation on the 2D torus are stable by perturbation by a potential. It will prove more delicate than one may assume a priori as the potential derails the local theory estimates on which the prove relies. This paper is in the spirit of previous works aiming to generalize the methods used on the cubic NLS such as \cite{fan20192ddefocusingnonlinearschrodingerequation} where the torus is taken irrational, and the main focus is in adapting the lattice counting arguments to the irrational case; \cite{tzvetkov2008invariantmeasuresdefocusingnls} where the problem is treated on the disc; as well as in a more general sense the general work on the sphere and Zoll manifolds in \cite{previos_burq_gerard,burq2025probabilisticwellposedenessnonlinearschrodinger}. For a fairly extensive list of such works see \cite{deng2012invariancegibbsmeasurebenjaminono}. Working with probabilistic initial data allows for an extension argument through an invariant measure aswell as for lower regularity of initial data. The initial data is taken almost surely in $H^{0^-}(\T^2)$ and the solution comes in the form of a regular ($H^s(\T^2), s>0$) perturbation of the solution to the linear Schrodinger equation. The outline of the argument is the same as for the potentialess equation, yet the introduction of the potential creates a great hurdle when trying to couple the local and global theory. In overcoming this we are able to gain insights on the essential factors that make Bourgain's proof, which a priori seems to strongly rely on the properties of the point-wise product of the exponential, a tool lost when approaching the equation with a potential. 
The problem of the NLSE with a potential 
\begin{equation}\label{eq:NLSE_2}
(i\partial_t  - \Delta  + V)u=- u|u|^2
\end{equation}
arises naturally as the first non trivial perturbation of the NLS on the $2D$ torus. Firstly we will develop a local Cauchy theory through a fixed point argument, this will prove tricky as our initial data is in $H^{0^-}(\T^2)$ and previous proofs of the regularizing character of the non linearity heavily rely on the particularities of $-\Delta$. Once the local theory is established we will extend the result through measure invariance and a bootstrap on the space of "good" initial data. As far as the deterministic terms are concerned the introduction of the potential does not play an important factor, although it does require some work to ensure that previous estimates still apply. It is in the probabilistic elements where the sturdiness of the previous methods comes under stress.
\subsection{Problem area and main result}
The first question that arises when introducing the potential is where it fits in the previously developed theory. This question is at the center of this paper and it boils down to wether we see the potential as contributing to the non-linearity or to the linear part of the equation. The distinction between the linear and non-linear parts of the equation are better seen when the equation is written under its Duhamel form. The first approach, taking $\phi$ as be the initial data, is to write (\ref{eq:NLSE_2}) as:
\begin{equation}\label{eq:Duhamel_Malo}
 u = e^{it(-\Delta)}\phi - i \int_{0}^t e^{i(t-s)(-\Delta)}(u|u|^2+Vu)ds.
\end{equation}
This approach seems a priori the easiest, as we take the potential as contributing to the nonlinearity and thus allows for a direct one to one with the previously established arguments in \cite{Bourgain1993,fan20192ddefocusingnonlinearschrodingerequation}. On the other hand a much more natural approach to the problem is to see $-\Delta+V$ as a linear operator replaccing the role of $-\Delta$. This translates into
\begin{equation}\label{eq:Duhamel_Bueno}
u = e^{it(-\Delta+V)}\phi - i \int_{0}^t e^{i(t-s)(-\Delta+V)}u|u|^2ds,
\end{equation}
but brings the disadvantage that we don't have a good description of the semigroup $e^{it(-\Delta+V)}$, it is this last approach that we will take. In order to see why we choose the second option we first have to talk about the influence of the potential on the initial data to be taken. The family of initial data that interests us both for physical and measure theory arguments are those following a normal distribution with covariance given by the resolvent of a self-adjoint operator. We can write explicitly these random variables. If we let $A^{-1}>0$ be a bounded, compact, self-adjoint operator on a Hilbert space $\HH$ (in our case $L^2(\T^2)$). Letting $\{e_n\}_{n\in\N}$ and $\{\lambda_n\}_{n\in\N}$ be an orthonormal basis of eigenfunctions and its corresponding eigenvalues, $\{g_n^w\}_{n\in\N}$ be independent and identically distributed complex gaussian variables centered in $0$ and of variance $1$, the following random variable 
    \begin{equation*}
        \Gg_{A}:\omega \longmapsto \sum_n \frac{g_n^\omega}{(\lambda_n)^{1/2}}e_n
    \end{equation*}
defines a normal distribution with covariance given by $A^{-1}$. The two natural candidates for $A$ in our case are $-\Delta$ and $-\Delta+V$, the choice once again depends on the role of the potential in the equation. Either way $\Gg_{A}(\omega)\in H^{0^-}$ almost surely. Our objective is to show that the integral term in the Duhamel form is regular, meaning that the solution will be the sum of a linear term, $e^{itA}\Gg_{A}(\omega)$,  and a regular term, $U$. This approach forces us to choose (\ref{eq:Duhamel_Bueno}) as we have that $V\Gg_{A}(\omega)\notin L^2(\T)$ almost surely. 

This consideration is also why we say that a potential (seen as an operator of order $0$) is the first non trivial alteration that we can do. For example in dimension $1$ the initial data is almost surely in $L^2(\T)$ and thus the (\ref{eq:Duhamel_Malo}) approach suffices, on the other hand taking a pseudo-differential operator of degree $\sigma<0$ would also pose a relatively trivial deviation from the proof in \cite{2D_Defocusing}. On the other hand, because of probabilistic considerations further discussed in section \ref{sect:measure_5} we take the Wick ordered nonlinearity $:u|u|^2:=u|u|^2 -2u\E \int|u|^2$, for a good definition and construction of the Wick ordering see \cite{Janson_1997}.

Before we introduce the main result let us briefly expose the hurdles introduced by the change $-\Delta\to-\Delta+V$ and how we will overcome them. The method we use to show that the integral term is regular is to develop the Wick ordered nonlinearity $\colon (e^{itA}\Gg_{A}(\omega)+U)|e^{itA}\Gg_{A}(\omega)+U|^2\colon$ into the sum of trilinear terms and then use the resonant interactions of the terms to bound the norms. In the case without a potential we are able to utilize the resonances thanks to the explicit interactions of the space fourier and time fourier, in the sense that 
\[e^{it(-\Delta)}e^{in\cdot x}e^{it(-\Delta)}e^{im\cdot x}=e^{i(|n|^2t+n\cdot x)}e^{i(|m|^2t+m\cdot x)}=e^{i\big((|n|^2+|m|^2)t+(n+m)\cdot x\big)}.\] 
This is used by Bourgain and others to bound the terms that contribute to the Sobolev norm\footnote{Actually to the Bourgain space norm, acting as surrogate for the Sobolev one.} by the size of lattice sets of the form:
\[\left\{(n_1,n_2,n_3)\in (\Z^2)^3 \ | \ \substack{n_2\neq n_1 \\ \ n_2\neq n_3 }, \substack{|n_i|\sim N_i \\ i=1,2,3}; \  |n_1|^2-|n_2|^2+|n_3|^2-|n_1-n_2+n_3|^2 =0 \right\}.\]

While these estimates are not important when it comes to the deterministic estimates as we can rely on Strichartz type estimates, they do play a crucial role in the bounding of the linear terms. This is why we introduce a pseudo differential operator $-\Delta+V^r$ whose eigenfunctions satisfy $e_ne_m\sim e_{n+m}$ in a sense to be explained later. This change makes it so we can construct the LWP argument for a random variable $\Gg_{-\Delta+V^r}(\omega)$, but in order to arrive at our intended result we will have to show that $\Gg_{-\Delta+V^r}(\omega)\sim\Gg_{-\Delta+V}(\omega)$ which will prove delicate. In order to extend our result we will need an almost surely finite Hamiltonian requiring us to take the the Wick ordered Hamiltonian. The Wick ordering of $\Gg_{-\Delta+V^r}(\omega)|\Gg_{-\Delta+V^r}(\omega)|^2$ grants a term that is equal to $0$ in the case of $-\Delta$ and is to the extent of my knowledge first treated here. It is this very term that is responsible for the solution in the sphere not being an alteration on the linear solution as seen in \cite{Burq_2025} and gives an understanding on some of the geometric considerations that allow for the well-possedness on the torus and not on the sphere.  We may now present the main result.
\begin{thm}\label{thm:thm1}
There is a $s>0$ and $\alpha>0$ such that for $V\in H^{2^+}(\T^2)$ and $\forall \kappa>0$ and $T\in \mathbb{R^+}$ the following Cauchy problem, 
\begin{equation}
    \begin{cases}
        (i\partial_t-\Delta+V)u = :u|u|^2:
        \\
        u(0,x)=\Gg_{-\Delta+V}(\omega)=\sum_{n} \frac{g_n^\omega}{(\lambda_n)^{1/2}}e_n(x).
    \end{cases}
\end{equation}
well posed, for $\omega$ outside of a set with measure smaller than $\kappa$. Meaning that we have unity and existence of a solution in the distributional sense. On top of that we have that letting $u^N$ be the solution of the truncated version (write $P_N:=\Pi_{-\Delta\le N^2}$, $u^N=P_N(u)$)
\begin{equation}
\begin{cases}
\Big(i\partial_t -\Delta  +V\Big)u^N =P_N(u^N|u^N|^2-2a_Nu^N)\\ 
u(0,x)=P_N(\Gg_{-\Delta+V}(\omega))(x)
\end{cases} 
\end{equation}
then
\begin{equation*}
U^N=u^N(t)-e^{2ic_N(\omega)}P_N\big(\Gg_{-\Delta+V}(\omega)\big)
\end{equation*}
converges in $L_{loc}^\infty H^s(\T^2)$ to
\begin{equation*}
U=u(t)-e^{2ic_\infty(\omega)}\sum_{n\in \mathbb{Z}^2} \frac{g_n(\omega)}{ (\lambda_n)^{1/2}}e^{i\lambda_n t}e_{n}(x)
\end{equation*}
at speed:
\begin{equation*}
    \Big\|U^N-U\Big\|_{L_{loc}^\infty H^s(\T^2)}\leq C(V,\kappa)N^{-\alpha},
\end{equation*}
with $a_N:=\E(\int |u^N|^2)$ and $c_N(\omega)=\int |u^N|^2-\E(\int |u^N|^2)$ and $\lim_{N\to\infty}c_N(\omega)=c_\infty(\omega)$ finite almost surely. 
\end{thm}
\subsection{Organization of the paper}
In section \ref{sect:def_and_tools_1} we introduce the tools that we will use along this paper, in particular we define the Bourgain spaces and a Strichartz type estimate that do the heavy lifting when it comes to the deterministic contributions later on. In this section we also present the probabilistic tools such as the hypercontractivity estimates that allow us to use the expectation of the random variables as a bounds. We introduce as well the Wiener measure associated to an operator and show that it is somewhat stable to a perturbation of order $0$. Next in section \ref{sect:Diagonaliation_2} we define $V^r$, a pseudo diagonalized version of $V$ in fourier, and study its eigenfunctions as well as the Wiener measure that it defines. In section \ref{sect:LWP_4} we establish LWP through a fixed point argument as well as the approximation argument by the truncated equation for initial data $\Gg_{-\Delta+V^r}(\omega)$, then thanks to results from the previous section it follows that we have LWP for $\Gg_{-\Delta+V}(\omega)$. Finally in section \ref{sect:measure_5} we prove theorem \ref{thm:thm1} by introducing the Gibbs measure as a weighted Wiener measure, and through its invariance by the flow and a pullback argument we conclude.

\subsection{Notation}
Finally beefore we jump in some words on notation. When writing $N,M,N^i,\dots$ They are to be understood as dyadic numbers. In the same way $n\sim N$ is in the dyadic sense $N/2\le |n|<N$. 
\begin{rem}
$\#\{n\in \Z^2, |n|\sim N\}\leq 4N^2.$
\end{rem}
When we work with the non linearity we will decompose in frequencies $n_1,n_2,n_3$ which correspond to $u_1,u_2,u_3$ and we will write them in growing module order $|n^1|\leq|n^2|\leq|n^3|$, the same applies to $N^1,N^2,N^3$. In general the subscript corresponds to the direct decomposition where the superscript corresponds to the ordered decomposition. This notation will be paired with projections, for a self-adjoint operator the projection $\Pi_{A<N}$ is defined by functional calculus. When working on the nonlinearity we will also use the notation $\delta(a=b)$ which is valued $1$ if true and $0$ if not.

We will write $0^+,0^-,2^+, \dots$ it refers to any number bigger, respectively smaller than the number in question. We note for $n\in\R^d$, $\la n  \ra:=(1+|n|^2)^{1/2}$ and the spaces $H^s, s\geq 0$ are the classical Sobolev spaces defined by $\|f\|^2_{H^s}=\sum_n\la n\ra^{2s}|\hat{f}_n|^2$ with $H^{s},s<0$ being defined by duality. When writing $\leeq, \gtrsim$ it is up to some constant that depends at most on $\|V\|_{H^{2^+}}$ and universal constants. Finally when we do the controls we will often write $N^\varepsilon$, this means that the control holds for any positive power of $N$. In a similar way we use $\alpha$ as some real number bigger than $0$.
\section{Definitions and tools}\label{sect:def_and_tools_1}
Let us begin by laying out the two set of tools that we will use in this paper, Bourgain spaces and probabilistic tools such as Wiener measures.
\subsection{Bourgain Spaces and Strichartz}\label{sect:bourgainspaces}
We now introduce the Bourgain spaces on which we will develop the local theory and that will act as surrogates for $L^{\infty}_t(H_A^s)$. For further reading and generalization of these spaces we recommend \cite{Tao2006NonlinearDE}.  This spaces allow us to apply Strichartz type estimates, they also allow for use of lattice counting estimates when bounding resonances. In the following $I$,$|I_\tau|=\tau$ are bounded time intervals.
\begin{defn}
Let $A$ be a self adjoint positive operator, let $b,s \in \mathbb{R}$ we note $X^A_{s,b}(\mathbb{R})=X_{s,b}(\mathbb{R})$ the closure of $\mathcal{C}^{\infty}_0(\mathbb{R},H_A^s(\T^2))$ in $L^2(\mathbb{R},L^2(\T^2))$ by the norm:
\[ \|e^{-itA}u(t,.)\|_{H^b(\R,H^s(\T^2)).}\]
For a given interval $I$ the space $X^A_{s,b}(I)$ is the restriction in time of the $X^A_{s,b}(\mathbb{R})$ space by the infimum norm.
\end{defn} 
The family of operators that interest us are perturbations of order $0$ of the Laplacian. As far as we are concerned we can take this to mean that our operators are of the form $A=-\Delta+\widetilde{V}$ with $\widetilde{V}$ symmetric, positive and $\|\widetilde{V}u\|_{H^s}\le C(\widetilde{V},s)\|u\|_{H^s}, \forall s\in \R $. We first introduce some known lemmas that show the equivalence between these operators when it comes to Sobolev spaces.
\begin{lem}\label{lemma:decro_vp}
Set $A=-\Delta+\widetilde{V}$, and let $\{\lambda_n\}_{n\in\Z^2}$ be its eigenvalues in growing order\footnote{Take a spiraling indexing of $\Z^2$. }. We have $\lambda_n\in \big[ \ |n|^2-C(\widetilde{V}),|n|^2+C(\widetilde{V}) \ \big]$,
With $\|\widetilde{V}u\|_{H^{1^+}}\le C(\widetilde{V})\|u\|_{L^2}$
\end{lem}
\begin{proof}
We have that in the sens of quadratic forms and through Sobolev injections:
\begin{equation*}
-\Delta-C(\widetilde{V})\leq -\Delta+\widetilde{V}\leq \Delta +C(\widetilde{V}).
\end{equation*}
We conclude by using Courant-Fischer.
\end{proof}
As stated before we take $V>0$ thus granting us $\lambda_n>0$. This means that $\{\lambda_n\}_{n\in\Z^2}$ behaves asymptotically as $\{|n|^2\}_{n\in\Z^2}$ meaning $\lim_{|n|\to \infty} \lambda_n/|n|^2=1$.
This spectral closeness means that $-\Delta+\widetilde{V}$ and $-\Delta$ define the same Sobolev spaces:
\begin{cor}
For all $s\in\R$ there exists a constant $C(\widetilde{V},s)$ such that $\forall u\in H^{s}(\T^2)$ we have:
\begin{equation*}
\frac{1}{C}\|u\|_{H_{-\Delta}^s}\leq \|A^{s/2}u\|_{L^2(\T^2)}\leq C\|u\|_{H_{-\Delta}^s}
\end{equation*}
meaning that the spaces $H_{-\Delta}^s(\T^2)$ and $H_{A}^s(\T^2)$\footnote{The $H_{A}^s(\T^2)$ defined by the norm $\|A^{s/2}\cdot\|_{L^2(\T^2)}$} are equivalent with equivalent norms.
\end{cor}
\begin{rem}
If $(-\Delta+\widetilde{V})^{-1}$ is compact we can write $u\in L^2\big(\T^2\times [0,\tau]\big)$ as
\begin{equation}\label{def:decompos}
u(x,t)=\sum_{n\in \N}\int_\mathbb{R} d\lambda {u}_n(\lambda)e_ne^{\lambda t}.
\end{equation}
With $\{e_n\}_{n\in\N}$ a hilbertian base of $L^2(\T^2)$. We have that its $X_{s,b}([0,\tau])$ norm may be written as 
\begin{equation}
\sum_n|\lambda_n|^{s}\int d\lambda \la\lambda-\lambda_n\ra^{2b}|u_n(\lambda)|^2.
\end{equation}
\end{rem}

In the same fashion we can show that for $|b|\le 1$ that the spaces $X^A_{s,b}$ and $X^{-\Delta}_{s,b}$ are equivalent.
\begin{lem}
    For an interval $|I|\le1$, there exists a constant $C(\widetilde{V},|I|)$ such that for $s\in \R$ and $|b|\le 1$
    \begin{equation*}
    \frac{1}{C}\|u\|_{X^{-\Delta}_{s,b}}\leq \|u\|_{X^A_{s,b}}\leq C\|u\|_{X^{-\Delta}_{s,b}}.
    \end{equation*}
\end{lem}
\begin{proof}
    For $b=0$ it follows directly from the previous lemma and the stability of the Schrödinger semigroup (with and without potential). It therefore suffices to show it holds for $b=1$, then $0<b<1$ follows by interpolation and $-1<b<0$ from duality. We have that for $f\in L^{2}(\T^2\times I)$
    \begin{equation*}
        i\partial_t(e^{itA}f)=e^{itA}(i\partial_t-A)f
    \end{equation*}
    and thus
    \begin{align*}
        & \|u\|^2_{X^{-\Delta+\widetilde{V}}_{s,1}} = \|\la-\Delta+\widetilde{V}\ra^{s/2}(i\partial_t-\Delta+\widetilde{V})u \|^2_{L^{2}(\T^2\times I)}+\|\la-\Delta+\widetilde{V}\ra^{s/2} f\|^2_{L^{2}(\T^2\times I)}\\
        & \|u\|^2_{X^{-\Delta}_{s,1}} = \|\la-\Delta\ra^{s/2}(i\partial_t-\Delta)u \|^2_{L^{2}(\T^2\times I)}+\|\la-\Delta\ra^{s/2} f\|^2_{L^{2}(\T^2\times I)},
    \end{align*}
    we are able to conclude by triangular inequality.
\end{proof}
Now we show that for $b>1/2$ we can use $X_{s,b}$ as a surrogate of $L^{\infty}_t(H_A^s)$.
\begin{lem}\label{lemma:Xsb}
Let $b>1/2$ we have that for $t\in I$:
\begin{equation}
\|u(t)\|_{H^s(\mathbb{T}^d)}\leeq \|u\|_{X_{s,b}(I)}
\end{equation}
up to some constant depending on $b$.
\end{lem}
This lemma can be seen as a consequence of the Sobolev embeddings of $H^b_t(I) \hookrightarrow C^0_t$ for $b>1/2$, and the fact that $e^{itA}$ preserves the $L^2(\T^2)$ norm. 
\begin{proof}
For a fixed $t\in I$ control $\|u(t)\|_{H^s}^2$:
\begin{align*}
&\sum_n \lambda_n^s\left| \int_\mathbb{R} d\lambda \hat{u}(n,\lambda)e^{i\lambda t}\right|^2 \leq \sum_n \lambda_n^s\left( \int_\mathbb{R} d\lambda |\hat{u}(n,\lambda)|\right)^2 
\\= & \sum_n \lambda_n^s\left( \int_\mathbb{R} d\lambda |\hat{u}(n,\lambda)|\frac{\la\lambda-\lambda_n\ra^{b}}{\la\lambda-\lambda_n\ra^{b}}\right)^2\leq \sum_n \lambda_n^s\int_\mathbb{R} d\lambda\la\lambda-\lambda_n\ra^{2b} |\hat{u}(n,\lambda)|^2\int_\mathbb{R} d\lambda\la\lambda-\lambda_n\ra^{-2b}\\
 \le & C(b)\sum_n \lambda_n^s\int_\mathbb{R} d\lambda\la\lambda-\lambda_n\ra^{2b} |\hat{u}(n,\lambda)|^2 \leq C(b) \|u\|_{X_{s,b}(I)}^2.
\end{align*}
\end{proof}
Finally we re-introduce the Strichartz type estimate shown in \cite{Bourgain1993} that will later take care of the deterministic elements when it comes to developing the local Cauchy theory. 
\begin{lem}\label{lem:Strichartz}
Let , $\varepsilon>0$, $f,g\in L^2(\T^2)$  be such that: $f(x)=\sum_{|n|\sim N} a_ne^{in\cdot x}$ and $g(x)=\sum_{|m|\sim M}b_me^{im\cdot x}$ then:
\begin{equation*}
\left\|(e^{-it\Delta}f) \ (e^{-it\Delta}g)\right\|_{L^2(\T^2\times I_\tau)} \leeq (N\wedge M)^\varepsilon \tau^\delta \left\|f\right\|_2\left\|g\right\|_2.
\end{equation*}
\end{lem}
\begin{proof}
We can suppose $\tau=1$ and from there get the result through variable change. Furthermore we can suppose without loss of generality that $N\leq M$ Write:
\begin{align*}
(e^{-it\Delta}f)(e^{-it\Delta}g)=&\left(\sum_{|n|\sim N}a_ne^{i(n\cdot x+|n|^2t)}\right)\left(\sum_{|m|\sim M}b_m e^{i(m\cdot x+|m|^2t)}\right) \\=& \sum_{|p|\le N+M} e^{ip\cdot x}\sum_{|n|\sim N} a_nb_{p-n}e^{i(|n|^2+|p-n|^2)t},
\end{align*}
applying Cauchy-Schwarz:
\begin{equation*}
\|(e^{-it\Delta}f)(e^{-it\Delta}g)\|_{L^2(\T^2 \times [0,1])}^2\leq \left\{ \max_{\substack{|p|\leq N+M \\ |j|\leq 2N^2}} \ \ \ r_{p,j}\right\}\left(\sum_n |a_n|^2\right)\left(\sum_m |b_m|^2\right),
\end{equation*}
where $r_{p,j}=\# \left\{n\in \mathbb{Z}^2 : |n|\leq N \text{ and } |n|^2+|p-n|^2=j\right\}$. This is bounded by the number of integer solutions to the equation:
\begin{equation*}
x_1^2+x_2^2=2j-|p|^2.
\end{equation*}
with $x_i=(2n_i-p_i)$. We can estimate it grossly by the number of integer solutions to the equation $x_1^2+x_2^2=R$ for some $R$ of size of order $N^2$. Thus the max is bounded thanks to a result from Hardy and Landau by $CN^{2\varepsilon}$ for all $\varepsilon >0$. Thus we arrive at the result.
\end{proof}
We can take the reciprocal inequality.
\begin{cor}
Let $1\le p\le 2$, $b=1^+-1/p$, $f\in X_{s,b}$. We have the control $\forall \varepsilon>0$  :
\begin{equation}\label{coro:stricartz}
\left\|f g\right\|_{L^p(\T^2\times I_\tau)} \leeq (N\wedge M)^\varepsilon \left\|f\right\|_{X_{0,b}}\left\|g\right\|_{X_{0,b}}.
\end{equation}
\end{cor}
\begin{proof}
    Once again it, for $p=1$ it is just Cauchy Schwartz, this means that we only need to prove it for $p=2$ as the rest is arrived at by interpolation. Write
    \begin{equation*}
        fg=\sum_{n,m}\int d\lambda d\mu \widehat{f_n}(\lambda+|n|^2)\widehat{g_m}(\mu+|m|^2)e^{i(n+m)\cdot x}e^{i(|n|^2+|m|^2+\lambda+\mu)t}
    \end{equation*}
    and thus control for $b>1/2$
    \begin{align*}
        \left\|f g\right\|_{L^2(\T^2\times I_\tau)}^2& \leeq\int d\lambda d\mu \frac{\left\|f\right\|_{X_{0,b}}\left\|g\right\|_{X_{0,b}}}{(\lambda\mu)^{2b}} \big\| \sum_{n,m} \widehat{f_n}(\lambda+|n|^2)\widehat{g_m}(\mu+|m|^2)e^{i(n+m)\cdot x}e^{i(|n|^2+|m|^2)t}\big\|\\
        & \leeq(N\wedge M)^\varepsilon \left\|f\right\|_{X_{0,b}}^2\left\|g\right\|_{X_{0,b}}^2.
    \end{align*}
\end{proof}
\subsection{Wiener measure and probabilistic estimates}\label{sect}
Now, introduce the probabilistic tools that we use in this paper. First we talk about the notion of Wiener measure. We will give a limited definition that suffices for our purposes, but a more in depth construction and explanation may be found in chapter 5 of \cite{derezinski2013mathematics}.
\begin{defn}
    Let $\HH$ be a hilbert space, $(A,\D(A))>0$, $\D(A)\subseteq \HH$, be a self adjoint operator such that $A^{-1}$ bounded and compact. Take $D^{-1}$ a trace class operator. Let $\{e_n\}_{n\in\N}$ and $\{\lambda_n\}_{n\in\N}$ be an orthonormal basis of eigenfunctions and its corresponding eigenvalues, let $\{g_n^w\}_{n\in\N}$ be independent and identically distributed complex gaussian variables centered in $0$ and of variance $1$, the following random variable 
    \begin{equation*}
        \Gg_{A}:\omega \longmapsto \sum_n \frac{g_n^\omega}{(\lambda_n)^{1/2}}e_n
    \end{equation*}
    defines a measure (through pullback) on $D^{1/2}\HH$ that is equivalent to the centered gaussian measure defined by taking $A^{-1}$ as the covariance. This measure corresponds to the Wiener measure associated to $A$ and we will note\footnote{There is no canonical choice of $D$} it $\rho_A=e^{-1/2\la\cdot,A^{-1}\cdot\ra_{\HH}}$.
\end{defn}
This gaussian measure endows us with some very useful estimates. For example taking $\HH=\Pi_{|n|\sim N}H^{0^-}$ and $A=\Pi_{|n|\sim N}(-\Delta+1)$ we get the following estimate due to hypercontractivity estimates on subgaussian random variables and Weyl's asymptotic estimates
\begin{equation*}
\mathbb{P}\left(\left\|\sum_{|n|\sim N} \frac{g_n^\omega}{(1+|n|^2)^{1/2}}e_n\right\|_2>t\right)\leq \exp(-t^2/(2\tr(A^{-1}))\leeq \exp(-t^2/(2N^\varepsilon)).
\end{equation*}
In the rest of the paper when we allude to controls granted by restricting $\omega$,  we are using the inequalities like the precedent one. The Wiener measure is somewhat stable under perturbation which will allow us to extend arguments from the "diagonalized" operator to $-\Delta+V$.
\begin{prop}[Feldman–Hájek]\label{prop:absolutecontinuitywiener}
    Let $\HH$ be a hilbert space, $D^{-1}$ be a trace class operator and let $A,B>0$ self adjoint operators such that $A^{-1}$ and $B^{-1}$ are bounded. We have that $\rho_A$ and $\rho_B$, defined on $D^{1/2}\HH$, are absolutely continuous with respect to each other if and only if the operator $A^{1/2}B^{-1}A^{1/2}-1$ is Hilbert–Schmidt on $\HH$. Furthermore if $\{\Pi_n\}_{n\in \N}$ is an increasing sequence of finite rank orthogonal projections in $\HH$ with strong limit $1$ we have that the Radon–Nikodym derivative , $\frac{d\rho_B}{d\rho_A}$, is equal to:
    \begin{equation}
        F(x)= \lim_{n\to+\infty}\det(\Pi_nA^{1/2}B^{-1}A^{1/2}\Pi_n)^{-1/2}\exp\Big(1/2\la x,\Pi_n(1-A^{-1/2}BA^{-1/2})\Pi_nx\ra_{\HH}\Big) 
    \end{equation}
    the limit taken in $L^1(D^{1/2}\HH,d\rho_A)$, and $F \geq 0$ $\|F\|_{L^1}=1$
\end{prop}
For a proof see chapter 11 \cite{derezinski2013mathematics}.
\section{Pseudo-differential "diagonalization" of the potential. and initial data}\label{sect:Diagonaliation_2}
We now introduce a pseudo-differential operator, $-\Delta+V^r$, that will allow us to establish local well posedness of the NLSE. We also show that the Wiener measures $\rho_{-\Delta+V}$ and $\rho_{-\Delta+V^r}$ are absolutely continuous with respect to each other which will allow us to use the LWP shown on the $-\Delta+V^r$ and the flow conserved measure properties of $-\Delta+V$. We first introduce the following partition of $\Z^2$ due to \cite{Bourgain1999}.
\subsection{Construction of $V^r$}
\begin{lem}\label{lem:Partition}
    For all $0<\rho<1/10$ there exist a $\gamma>0$, $0<\delta<\rho$ and a partition of $\Z^2$, $\{\Omega_\alpha\}_{\alpha\in A}$ such that:
    \begin{enumerate}
        \item $\forall\alpha\in A$ and $n,n'\in \Omega_\alpha$ $|n-n'|+||n|^2-|n'|^2|<\gamma+|n|^\rho$ 
        \item $\forall\alpha,\beta\in A$ such that $\alpha \neq \beta$ and $n\in \Omega_\alpha, n'\in \Omega_\beta$ $|n-n'|+||n|^2-|n'|^2|>|n|^{\delta}$
    \end{enumerate}
\end{lem}
\begin{rem}
    The $\rho$ and $\delta$ are the limiting factors when we want to gain regularity. Meaning that if we where to get a higher $\delta$ and a lower $\rho$ the difference between the linear evolution and the solution would be shown to be more regular and we would be able to approximate $s\le 1/2$.
\end{rem}
Denote in the following for $\alpha\in A$
\begin{equation*}
    \Pi_\alpha := \sum_{n\in\Omega_\alpha}\Pi_n.
\end{equation*}
We can give a good description of such sets.
\begin{prop}\label{prop:descripcion_Omega}
    Let $\{\Omega_\alpha\}_{\alpha\in A}$ be a partition of $\Z^2$ as described above, for a given $\alpha$ choose a element $m\in \Omega_\alpha$ such that $|m|\leq |k|$ for all $k\in \Omega_\alpha$, we have that for $|m|$ big enough that any element $k\in\Omega_\alpha$ can be written as:
    \begin{equation}
        k=m+\lambda_1 u_1+\lambda_2 u_2 \ \ \ \ \ (\lambda_1,\lambda_2)\in[-|m|^{\rho-1},|m|^{\rho-1}]\times[-2|m|^{\rho/2},2|m|^{\rho/2}]
    \end{equation}
    Where $u_1=m/|m|$ and $u_2$ is a unitary vector that completes an orthonormal basis of $\R^2$ together with $u_1$.
\end{prop}
This means that wee can see the $\Omega_{\alpha}$ as tangent segments to the circle with center at the origin and of length $|m|^{\rho}$ as seen in the following figure.
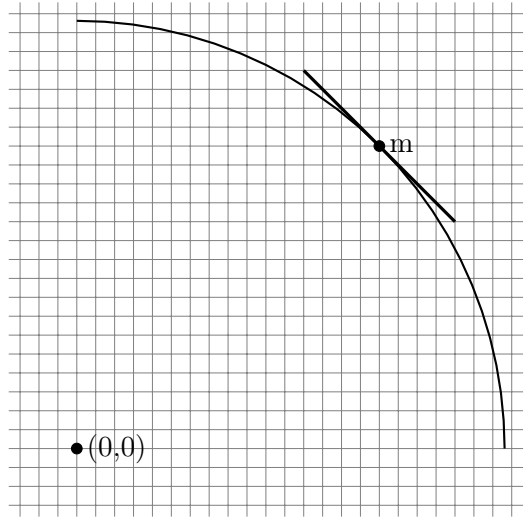
\begin{figure}[h]
\centering
\begin{tikzpicture}
\draw[step=0.25cm,gray,very thin] (-2.9,-2.9) grid (3.9,3.9);
\filldraw[black] (-2,-2) circle (2pt)node[anchor=west]{(0,0)};
\filldraw[black] (2,2) circle (2pt)node[anchor=west]{m};
\draw [black,thick,domain=0:90] plot ({-2+5.656*cos(\x)}, {-2+5.656*sin(\x)});
\draw [black,very thick,domain=-1:1] plot ({2+\x}, {2-\x});
\end{tikzpicture}
\end{figure}

\begin{proof}
    The fact that we can choose an $m$ comes from $\Omega_\alpha$ being finite in size and by construction we can write uniquely $k$ as $m+\lambda_1 u_1+\lambda_2 u_2$. We are to proof that $(\lambda_1,\lambda_2)\in[-|m|^{\rho-1},|m|^{\rho-1}]\times[-|m|^{\rho/2},|m|^{\rho/2}]$. We have (for $|m|$ big enough) that for $k\in \Omega_\alpha$ that $|k-m|+||k|^2-|m|^2|\leq2|m|^\rho$. On the other hand $|k-m|^2=\lambda_1^2+\lambda^2_2$ and $|k|^2-|m|^2=||m|+\lambda_1|^2+\lambda^2_2-|m^2|=\lambda_1^2+\lambda^2_2+2\lambda_1|m|$. This means first that $|\lambda_1|,|\lambda_2|\leq 2|m|^{\rho/2}$. Now let us look at two different cases, first take $\lambda_1>0$ in that case we have that $\lambda_1|m|\leq |m|^\rho$. If we have that\footnote{This case boils down to estimates on the asymptotic behavior of short enough ring segments.} $\lambda_1\leq 0$ then we have that by choice of $k$ that $|k|^2=||m|+\lambda_1|^2+\lambda^2_2\geq |m|^2 $ implies that $|m|^\rho\geq |\lambda_1||m|$ concluding the argument.
\end{proof}
\begin{rem}\label{rem:buen indice de omega}
    This together with the fact presented in \cite{1077308005} that the number of points in $\Z^2$ on an arc of a circle of radius $R$ and length smaller than $CR^{1/3}$ is at most two, means that we can index (up to duplicates) the elements in $\Omega_\alpha$ by their modulus. It also means that the map $\lambda_1\mapsto \lambda_2$ is bijective; 
\end{rem}
We define, for $\sigma\in\R$, $\Ll^\sigma$  the space of operators, $Q:\CC^\infty(\T^2)\rightarrow \D'(\T^2)$ such that for all $M\in \N$ there exists a $C_M\in \R^+$ that satisfies for all $n,n'\in \Z^d$:
\begin{equation*}
    \Big\|\Pi_n Q \Pi_{n'}\Big\|_{\Ll(L^2(\T^2))}\le C_M\la|n|+|n'|\ra^{\sigma}\la n-n'\ra^{-M}
\end{equation*}
and\footnote{In the sense that $Q$ admits a "constant" extension to operate on $\CC^\infty(\R_t\times \T^2)$} $[Q,\partial_t]=0$. Clearly we have that if $Q\in \Ll^\sigma $ then $Q$ is a bounded operator fro $H^s\to H^{s+\sigma}$ for $s\in\R$. The following proposition is an adaptation of the arguments presented in \cite{equaequivalente}. In the following fix $1/10>\rho>0$ and take $0< \delta<\rho$ satisfying lemma \ref{lem:Partition}.
\begin{prop}\label{prop:comutadorlaplmasv}
For some $\delta>0$ and $V\in H^{2^+}(\T^2)$ there exists $Q\in \Ll^{-\delta} $, $\|Q\|_{\Ll(L^2(\T^2))}<1$ such that
\begin{equation*}
    (1+Q)(i\partial_t-\Delta+V)-(i\partial_t-\Delta+V^r)(1+Q)=R.
\end{equation*}
    Where $R\in \Ll^{-\delta}$ and letting $|M(\alpha)|$ the upper bound of $\Omega_\alpha$:
\begin{equation*}
V^r:= \Big(\sum_{\substack{\alpha, \\M_\alpha^{\delta}<C\|V\|^2_{H^{2^+}}}}\Pi_\alpha \Big) V \Big(\sum_{\substack{\alpha, \\M_\alpha^{\delta}<C\|V\|^2_{H^{2^+}}}}\Pi_\alpha \Big) +\sum_{\substack{\alpha, \\M_\alpha^{\delta}\ge C\|V\|^2_{H^{2^+}}}}\Pi_\alpha V \ \Pi_\alpha. 
\end{equation*}
\end{prop}
\begin{proof}
It suffices to show that, noting $V^a=V-V^r$, the operator defined by 
\[\Pi_n Q \Pi_{n'}:=\frac{\mathbbm{1}_{||n|^2-|n'|^2|>(|n|+|n'|)^{\delta}/4}}{|n|^2-|n'|^2}\Pi_n V^a \Pi_{n'},\]
 satisfies our conditions. First we see that $Q\in \Ll^{-\delta} $ by construction. On the other hand the Sobolev embedding $H^{1+\varepsilon}(\T^2)\hookrightarrow L^\infty(\T^2) $ yields that
$\|Q\|_{\Ll(L^2(\T^2))}<1$ as $\|\Pi_{|n|>N}V\|_\infty\leeq N^{-1}\|V\|_{H^{2+\varepsilon}}^2$. We are left with showing that:
\begin{equation*}
    (1+Q)(i\partial_t-\Delta+V)-(i\partial_t-\Delta+V^r)(1+Q)\in \Ll^{-\delta},
\end{equation*}
which is reduced to showing that
\begin{equation*}
    V^a+[Q,-\Delta]+QV-V^rQ\in  \Ll^{-\delta}.
\end{equation*}
On the one hand we clearly have that $\Ll^{\sigma_1}\circ\Ll^{\sigma_2}\subseteq\Ll^{\sigma_1+\sigma_2}$ yielding that $QV-V^rQ\in  \Ll^{-\delta}$. Finally we study for $(|n|+|n'|)^{\delta}>\|V\|_{H^{2^+}}$
\[\Pi_n\big(V^a+[Q,-\Delta]\Big)\Pi_{n'}.\]
Start by realizing that this expression is $0$ if $n,n'\in \Omega_\alpha$. If $||n|^2-|n'|^2|>(|n|+|n'|)^{\delta}/4$ then $\Pi_n[Q,-\Delta]\Pi_{n'}=\Pi_nV^a \Pi_{n'}$ and thus $\Pi_n\big(V^a+[Q,-\Delta]\Big)\Pi_{n'}=0$, on the other hand if $n\in\Omega_\alpha,n'\in \Omega_\beta$, $\alpha\neq\beta$ and $||n|^2-|n'|^2|\le(|n|+|n'|)^{\delta}/4$ then $|n-n'|>(|n|+|n'|)^{\delta}/2$ and thus $\|\Pi_nV^a \Pi_{n'}\|_{\Ll(L^2(\T^2))}\le C(V)(|n|+|n'|)^{\delta}$ which concludes the proof.
\end{proof}
Taking advantage of the description of the $\Omega_\alpha$ we can study the operator $-\Delta+V^r$ its eigenfunctions. We have thanks to Rellich-Kato that $-\Delta+V^r$ is self adjoint. On the other hand we have that both $-\Delta$ and $-\Delta+V^r$ commute with the family of projectors $\{\Pi_\alpha\}_{\alpha\in A}$ and thus thanks to the min-max formula we have that there is a hilbertian basis of $L^2(\T^2)$ of eigenfunctions of $-\Delta+V^r$, $\{e_n^r\}_{n\in\Z^2}$, with $\{\lambda_n^r\}_{n\in\Z^2}$ being their associated eigenvalues, such that there exists a constant $C(V)>|\lambda_n^r-|n|^2|$ and such that (provided $|n|$ is big enough) $\Pi_{\alpha(n)}e_n^r=e_n^r$ and for $\beta\neq \alpha(n)$ we get $\Pi_\beta e_n^r=0$. In the following fix such a basis, we will write the decomposition of its eigenvectors in Fourier as:
\begin{equation*}
e_n^r(x)=\sum_{k\in\Omega_{\alpha(n)}}\chi_n^ke^{ik\cdot  x}
\end{equation*}
and by orthonormality it follows that:
\begin{equation*}
    e^{in\cdot x}=\sum_{k\in\Omega_{\alpha(n)}}\overline{\chi_k^n}e_k^r(x)
\end{equation*}
The following lemma gives a good description of the $|\chi_n^k|$.
\begin{lem}\label{lem:control_e_f}
There exists a $C(V)$ such that for all $n\in\Z^2$:
\begin{align}
    \sum_{k\in\Omega_{\alpha(n)}} \la |n|^2-|k|^2 \ra^2\big|\chi_n^k\big|^2<C(V)\\
    \sum_{k\in\Omega_{\alpha(n)}} \la |n|^2-|k|^2 \ra^2\big|\chi_k^n\big|^2<C(V)
\end{align}
\end{lem}
\begin{proof}
This holds trivially for if $|n|\leq C(V)$, thus suppose that $|n|$ is big enough, this allows us to only work with the partition preserved elements. We have that $V^r=(-\Delta+V^r)+(-\Delta)$ and thus
\begin{equation*}
    V^re_n^r=\sum_{k\in\Omega_{\alpha(n)}}(\lambda_n-|k|^2)\chi_n^ke^{ik\cdot x}
\end{equation*}
    Taking the $L^2$ norm on both sides and using $C(V)>|\lambda_n-|n|^2|$, that $V^r$ is bounded, and $\sum_k|\chi_n^k|^2=\|e_n^r\|^2_2=1$ we conclude for the first inequality, as for the second the same process, using that the $\chi_n$ are symmetric, suffices.
\end{proof}
We can rewrite this making use of out description of $\Omega_\alpha$ granted by proposition \ref{prop:descripcion_Omega}. Let $n$ be fixed, write $k(\lambda_1,\lambda_2)=n+\lambda_1 u_1+\lambda_2 u_2\in\Z^2$ with $u_1=n/|n|$ and $u_2$ is a unitary vector that completes an orthonormal basis of $\R^2$ together with $u_1$. The previous inequality yields the following inequality 
    \begin{equation}\label{eq:buenadescrippciondepsi}
        \sum_{k(\mu_1,\mu_2)\in\Omega_{\alpha(n)}}|\chi_n^{k(\mu_1,\mu_2)}|^2\la\ \mu_2\ra^4 \le C(V).
    \end{equation}
The inequality comes down to Pythagoras. Suppose by symmetry that $|k(\mu_1,\mu_2)|>|n|$. We can look at the triangle formed by the origin and $p,n$, we have that $\la|n|^2-|k(\mu_1,\mu_2)|^2\ra\ge \la n-k(\mu_1,\mu_2) \ra^2\geq \la \mu_2\ra ^2$.
This concentration of the fourier transform also grants us a pointwise control of the $e_n^r$, we can see easily that the $\mu_2$ are uniquely defined by the $\mu_1$, and therefore:
\begin{equation*}
    |e_n^r(x)|^2\leq \Big|\sum_{k(\mu_1,\mu_2)\in\Omega_{\alpha(n)}}\chi_n^{k(\mu_1,\mu_2)} e^{i(k(\mu_1,\mu_2)\cdot x)} \Big|^2\le \sum_{k(\mu_1,\mu_2)\in\Omega_{\alpha(n)}}|\chi_n^{k(\mu_1,\mu_2)}|^2\la\mu_2\ra^4\sum_{\mu_2}\la\mu_2\ra^{-4}
\end{equation*}
and this last term is bounded by $C(V)$.
\subsection{Initial data and absolute continuity}\label{subsection:proba}
First it is clear that all three $\Gg_{-\Delta}(\omega),\Gg_{-\Delta+V}(\omega)$ and $\Gg_{-\Delta+V^r}(\omega)$ are almost surely in $H^{0^-}=\bigcap_{\varepsilon>0} H^{-\varepsilon}$. We show this by taking hypercontractivity estimates and realizing that the eigenvalues of the three are asymptotically equivalent to $|n|^2$, thus the trace of the resolvent of the truncated operators grows as $\log N$. On the other hand having seen that the $e_n^r$ are uniformly bounded in module we get a pointwise estimate on the dyadic decompositions of $\Gg_{-\Delta+V^r}$.
\begin{prop}\label{prop:norma_gauss_infty}
We have that for $c>0$ outside of a set of measure $N^{-c}$ that
\begin{equation}\label{eq:norma_gaus_infty}
\left\|\sum_{|n|\sim N}\frac{g_n^\omega}{\lambda_n^r} e_n^r\right\|_\infty \leq \log N.
\end{equation}   
\end{prop}
\begin{proof}
This is because once again using hypercontractivity estimates we have that for a fixed $x\in \T^2$ and $t>0$:
\begin{equation*}
\mathbb{P}\left(\left|\sum_{|n|\sim N}\frac{g_n^\omega}{\lambda_n^r} e_n^r\right|\geq t\right)\leq \exp\left( {-ct^2}\left(\E\left|\sum_{|n|\sim N}\frac{g_n^\omega}{\lambda_n^r} e_n^r\right|^2\right)^{-1}\right)\leq e^{-ct^2/\log N}.
\end{equation*}
And therefore taking $t=\log N$ we get our estimation.
\end{proof}
On the other hand we can show that the Wiener measures associated are somewhat stable by perturbation of order 0.
\begin{prop}
    Fix $\varepsilon>0$, let $A$ be a self adjoint operator $\Ll^2(\T^2)$ and $B$ a relatively self adjoint operator with respect to A in $\Ll^0(\T^2)$, suppose that $A,A+B>0$. $A $ and $A+B$ define Wiener measures over $H^{-\varepsilon}$ and $\rho_{A}$ and $\rho_{A+B}$ are absolutely continuous with respect to each other.
\end{prop}
\begin{proof}
    Set $\HH={L^2(\T^2)}$ and realize that $A^{-1}$ and $(A+B)^{-1}$ are bounded on $\HH$. Now we look at proposition \ref{prop:absolutecontinuitywiener} with $D=(-\Delta)^{1+2\varepsilon}$. We have that $A$ and $A+B$ define wiener measures on $H^{-1-\varepsilon}$ and we can restrict them to $H^{-\varepsilon}(\T^2)$ as we have seen previously through hypercontractivity estimates that $\rho_{A,A+B}(H^{-1-\varepsilon}(\T^2) \textbackslash H^{-\varepsilon}(\T^2) )=0$
    For the absolute continuity we have to prove that $A^{1/2}(A+B)^{-1}A^{1/2}-1$ is Hilbert-Schmidt on $L^2(\T^2)$. First realize that:
    \begin{equation*}
        (A+B)^{-1}A=1-(A+B)^{-1}B \implies (A+B)^{-1}=A^{-1}-(A+B)^{-1}BA^{-1}
    \end{equation*}
from which:
\begin{equation*}
        A^{1/2}(A+B)^{-1}A^{1/2}-1= -A^{1/2}(A+B)^{-1}BA^{-1/2}\in \Ll^{-2}. 
\end{equation*}
If we take an orthonormal basis of, say $\{e_n\}_{n\in\Z^2}=\{e^{in\cdot x}\}_{n\in \Z^2 }$ we have:
\begin{equation*}
    \sum_{n\in\Z^2}\| \big(A^{1/2}(A+B)^{-1}A^{1/2}-1\big)e_n\|^2_{L^2}\leeq  \sum_{n\in\Z^2} \frac{1}{|n|^4} < +\infty
\end{equation*}
meaning that $A^{1/2}(A+B)^{-1}A^{1/2}-1$ is Hilbert-Schmidt. The absolute continuity follows directly from there.
\end{proof}
If we take $A=-\Delta+V^r$ and $B=V^a$ we have that $\rho_{-\Delta+V}$ and $\rho_{-\Delta+V^r}$ are absolultley continuous with respect to each other and the RN derivative $\frac{d\rho_{-\Delta+V}}{d\rho_{-\Delta+V^r}}$ is 
\begin{align*}
    &\lim_{N\to\infty}\det(\Pi_{|n|\le N}(-\Delta+V^r)^{1/2}(-\Delta+V)^{-1}(-\Delta+V^r)^{1/2}\Pi_{|n|\le N})^{-1/2}\times\\
    &\exp\Big(1/2\la x,\Pi_{|n|\le N}(-\Delta+V^r)^{-1/2}V^a(-\Delta+V^r)^{-1/2}\Pi_{|n|\le N}x\ra_{L^2}\Big).
\end{align*}
The determinant is bounded by a constant depending on $\|V\|_\infty$ as it is the Fredholm determinant of a bounded operator. On the other hand for $x\in L^2(\T^2)$, $N>0$ we have a constant $C(V)>0$
\begin{equation*}
    \la x,\Pi_{|n|\le N}(-\Delta+V^r)^{-1/2}V^a(-\Delta+V^r)^{-1/2}\Pi_{|n|\le N}x\ra_{L^2}\leq C\|x\|^2_{H^{-1}}
\end{equation*}
this means that letting $\Lambda_\lambda=\{u\in \D', \big\|u\big\|_{H^{-1}}< \lambda\}$, that for $M\subset \Lambda_\lambda$
\begin{equation*}
    \rho_{-\Delta+V}(M)e^{-C \lambda^2} \le   \rho_{-\Delta+V^r}(M) \le \rho_{-\Delta+V}(M) e^{C\lambda^2},
\end{equation*}
on the other hand we have both through hypercontractivity estimates $\rho_{-\Delta+V}\big((\Lambda_\lambda)^c\big)\le e^{-C\lambda^2}$ and $\rho_{-\Delta+V^r}\big((\Lambda_\lambda)^c\big)\le e^{-C'\lambda^2}$. Thus if for a set $B\subset H^{-\varepsilon}$ we have that, for some $\tau,\delta>0$ $\rho_{-\Delta+V^r}\big(B\big)\le e^{-1/\tau^\delta}$ then for $\lambda>0$:
\begin{equation*}
  \rho_{-\Delta+V}\big(B\big)\le \rho_{-\Delta+V^r}\big(B\cap\Lambda_\lambda\big) +\rho_{-\Delta+V^r}\big(B\cap(\Lambda_\lambda)^c\big)\le  e^{C \lambda^2-1/\tau^\delta} + e^{-C'\lambda^2}\leeq e^{-C''/{\tau^\delta}}
\end{equation*}
with a choice of $\lambda^2 = C/({2\tau^\delta})$. This will be useful later allowing to extend the probabilistic estimates from $-\Delta+V^r$ to $-\Delta+V$.
\section{Local Well Posedness}\label{sect:LWP_4}
\subsection{Local wellpossednes with respect to the Wiener measure }
We work on the LWP theory.
\begin{prop}\label{prop:LWP}
For some $0<s$ and $\forall\tau>0$ small enough, there exists a set $\Omega$ and $\alpha>0$ such that, $\rho_{-\Delta+V^r}(\Omega^c)\leeq e^{-1/\tau^\alpha}$ and for all $\omega\in\Omega$ we have that
\begin{equation}\label{eq:Non_Linear_altered_equation}
    \begin{cases}
        (i\partial_t-\Delta+V)u = :u|u|^2:
        \\
        u(0,x)=\Gg_{-\Delta+V^r}(\omega)=\sum_{n} \frac{g_n^\omega}{(\lambda_n^r)^{1/2}}e_n^r(x)
    \end{cases}
\end{equation}
is weakly wellpossed. The solution is the limit of the solutions to the truncated problems in the sense that 
\begin{equation*}
    u_N-e^{2ic_N^r(\omega)t}e^{it(-\Delta+V)}P_{N}\Gg_{-\Delta+V^r}(\omega)\xrightarrow{N\to\infty} u-e^{2ic_\infty^r(\omega)t}e^{it(-\Delta+V)}\Gg_{-\Delta+V^r}(\omega)
\end{equation*}
converges in $L^{\infty}_t(I,H^s(\T^2))$, the truncated equation being:
\begin{equation}
\begin{cases}
\Big(i\partial_t - \Delta  +V\Big)u^N =P_N(u^N|u^N|^2-2a_N^ru^N)\\ 
u(0,x)=P_N(\Gg_{-\Delta+V^r}(\omega))(x)=\sum_{(\lambda_n)^{1/2}\leq N} \frac{g_n^\omega}{(\lambda_n^r)^{1/2}}e_n^r(x)
\end{cases} 
\end{equation}
With $P_N=\Pi_{-\Delta+V^r\leq N^2}$ or equivalently $\Pi_{-\Delta\leq N^2}$,
\begin{equation}\label{eq:an}
a_N^{(r)}:=\mathbb{E}\Bigg(\left\|P_N(\Gg_{-\Delta+V^{(r)}}(\omega))(x)\right\|^2_{L^2(\T^2)}\Bigg).
\end{equation}
and
\begin{equation}\label{eq:cw}
c^{(r)}_N(\omega)=\int |u^N|^2-a_N^{(r)}.
\end{equation}
\end{prop}
Knowing that $\rho_{-\Delta+V^r}$ and $\rho_{-\Delta+V}$ are almost $L^\infty$ with respect to each other it follows that:
\begin{cor}\label{prop:LWP_laplaciana}
For some $0<s$ and $\forall \tau>0$ small enough, there exists a set $\Omega$ such that, $\rho_{-\Delta+V}(\Omega^c)\leeq e^{-1/\tau^\alpha}$ and for all $\omega\in\Omega$ we have that
\begin{equation*}
    \begin{cases}
        (i\partial_t-\Delta+V)u = :u|u|^2:
        \\
        u(0,x)=\Gg_{-\Delta+V}(\omega)=\sum_{n} \frac{g_n^\omega}{(\lambda_n)^{1/2}}e_n(x)
    \end{cases}
\end{equation*}
is weakly wellpossed and the solution is the limit of the solutions to the truncated problems in the sense that 
\begin{equation*}
    u_N-e^{2ic_N(\omega)t}e^{it(-\Delta+V)}\Pi_{|n|\leq N}\Gg_{-\Delta+V}(\omega)\xrightarrow{N\to\infty} u-e^{2ic_\infty(\omega)t}e^{it(-\Delta+V)}\Gg_{-\Delta+V}(\omega)
\end{equation*}
converges in $L^{\infty}_tH^s(\T^2)$, the truncated equation being:
\begin{equation*}
\begin{cases}
\Big(i\partial_t -\Delta  +V\Big)u^N =P_N(u^N|u^N|^2-2a_Nu^N)\\ 
u(0,x)=P_N(\Gg_{-\Delta+V}(\omega))(x)
\end{cases} 
\end{equation*}
\end{cor}
For a given initial datum there is a solution if and only if there is a unique fixed point to the equation written under its Duhamel form, $Tu=u$, with $T$ being the transform:
\begin{equation}\label{eq:Duhamel_transorm}
T: u \longmapsto e^{it(-\Delta+V)}\Gg_{-\Delta+V^r}(\omega) - i \int_{0}^t e^{i(t-s)(-\Delta+V)}:u|u|^2:ds.
\end{equation}
The idea of the proof if to show that $T$ is a contraction on $e^{it(-\Delta+V)}\Gg_{-\Delta+V^r}(\omega)+B_{X_{\delta,b}}(0,1)$ for $0<\tau$ small enough and $b>1/2$. Thus we want to show that on a short enough time span the integral term in (\ref{eq:Duhamel_transorm}) sends (up to restriction in $\omega$) $e^{it(-\Delta+V)}\Gg_{-\Delta+V^r}(\omega)+B_{X_{\delta,b}}(0,1)\to B_{X_{\delta,b}}(0,\tau)$, with $B_{X_{\delta,b}}(0,\tau)$ the  ball in ${X_{\delta,b}}$ of center $0$ and radius $\tau$.  First we introduce the following lemma which works in a more general setting and allows us to "gain" a derivative in time.
\begin{lem}\label{lemma:estimacion_1_dualidad}
Write $S(t)=e^{it(-\Delta+V)}$, we have for some $b > 1/2 >b'>0$ such that $b+b'<1$, that for $\tau <1$:
\begin{equation}\label{eq:norma}
\left\|\int_{0}^t S(t-s)\Gamma(u)ds\right\|_{X_{s,b}([0,\tau])}\leq C\tau^\alpha \left\|\Gamma(u)\right\|_{X_{s,-b'([0,\tau])}}
\end{equation}
for some $\alpha=1-b-b'>0$. Where $\Gamma(u)$ is a function of $u$ eventually either $u|u|^2$ or $P_N(u^N|u^N|^2)$.
\end{lem}
\begin{proof}
Fix $0<\tau<1$ and let $I=[0,\tau]$, we follow the estimation done for the 1-dimensional case in \cite{bourgain1999global} and the method presented in \cite{SB_1994-1995__37__163_0}. First recall that:
\begin{equation*}
\|u\|_{X_{s,b}(I)}=\|S(-t)u(t,.)\|_{H^b(I,H^s(\T^2))}.
\end{equation*}
Thus (\ref{eq:norma}) is equivalent to proving:
\begin{equation}\label{eq:norma_sin_s}
\left\|\int_{0}^t S(-s)\Gamma(u)ds\right\|_{H^b(I,H^s(\T^2))}\leq C\tau^\alpha \left\|\Gamma(u)\right\|_{X_{s,-b'([0,\tau])}}.
\end{equation}
Let us rewrite the integral term in (\ref{eq:norma_sin_s}) as:
\begin{equation*}
\chi(t)\sum_n   e_n(x) \int_\mathbb{R} d\lambda \widehat{\Gamma(u)}(n,\lambda)\frac{e^{it\lambda}-1}{i\lambda }.
\end{equation*}
Where $\chi(t)$ is a smooth indicator function valued $1$ on $I$ and compactly supported in a neighborhood of $I$ that we will use to estimate the norm later through localization in time. $\widehat{\Gamma(u)}(n,\lambda)$ is the Fourier transform in time of the coefficients of the decomposition in eigenfunctions of $\Gamma(u)$. This can be rewritten\footnote{We lose the $i$.} as the sum of:
\begin{align}
&\chi(t)\sum_n e_n(x)\int_{|\lambda-\lambda_n| > 1/\tau}d\lambda \widehat{\Gamma(u)}(n,\lambda)\frac{e^{it(\lambda-\lambda_n)}}{\lambda - \lambda_n}\label{eq:term1sb}\\
-&\chi(t)\sum_n e_n(x)\int_{|\lambda-\lambda_n| > 1/\tau}d\lambda \widehat{\Gamma(u)}(n,\lambda)\frac{1}{\lambda - \lambda_n}\label{eq:term2sb}\\
&\chi(t)\sum_ne_n(x) \int_{|\lambda-\lambda_n| < 1/\tau}d\lambda \widehat{\Gamma(u)}(n,\lambda)\frac{e^{it(\lambda -\lambda_n)}-1}{\lambda - \lambda_n}\label{eq:term3sb}.
\end{align}
Let us begin by estimating the norm of (\ref{eq:term2sb}):
\begin{equation*}
\left\|(\ref{eq:term2sb})\right\|_{H^{b}(I,H^s(\T^2))}  
\leeq \|\Gamma(u)\|_{X_{s,-b'([0,\tau])}}\|\mathbbm{1}_{|\lambda|> 1/\tau}|\lambda|^{b+b'-1}\|_\infty\leeq \tau^\alpha \|\Gamma(u)\|_{X_{s,-b'([0,\tau])}}.
\end{equation*}
The term (\ref{eq:term1sb}) is controlled the same but using the fact that we have taken $\tau<1$. Finally when it comes to (\ref{eq:term3sb}) we can consider the series decomposition:
\begin{equation*}
(\ref{eq:term3sb})=\sum_{k\geq 1}\frac{(i t)^k}{k!}\sum_n \int_{|\lambda-\lambda_n| \leq 1/\tau}d\lambda \widehat{\Gamma(u)}(n,\lambda)(\lambda - \lambda_n)^{k-1}\chi(t).
\end{equation*}
We can control as above $\|(\ref{eq:term3sb})\|_{H^{b}(I,H^s(\T^2))}$ by:
\[\sum_{k\geq 1} \tau^{1-k} \left\|\frac{t^k}{k!}\chi\right\|_{H^b(I,H^s(\T^2))}\|\Gamma(u)\|_{X_{s,-b'([0,\tau])}}\|\mathbbm{1}_{|\lambda|>1/\tau}|\lambda|^{b'}\|_{L^2(\mathbb{R})}\leq C\tau^\alpha\|\Gamma(u)\|_{X_{s,-b'([0,\tau])}}.\]
\end{proof}
We are now ready to prove proposition \ref{prop:LWP}. 
\begin{proof}[Proof of proposition \ref{prop:LWP}]
We have seen that thanks to (\ref{eq:norma}) that in order to have uniqueness and existence of a solution to (\ref{eq:Non_Linear_altered_equation}) on $[0,\tau]$ it suffices to show that for $\omega$ outside of a set of size $e^{-1/\tau^\alpha}$ that for $u\in e^{it(-\Delta+V)}\Gg_{-\Delta+V^r}(\omega)+B_{X{s,b}}(0,1)$ for $b > 1/2 >b'>0$ such that $b+b'<1$
    \begin{equation*}
        \Big\|:u|u|^2:\Big\|_{X^{-\Delta+V}_{s,-b'}([0,\tau])}\leq C\tau^{\alpha'}
    \end{equation*}
as this would mean that $T$ is a contraction on $e^{it(-\Delta+V)}\Gg_{-\Delta+V^r}(\omega)+B_{X{s,b}}(0,1)$ and therefore there is a unique fixed point that satisfies $Tu=u$. The approximation considerations will follow from the estimations done in order to show that $T$ is indeed a contraction. We first we must adapt the problem so that the nonlinear term's interaction in space may be developed in a suitable manner. We can show that the evolution semi-groups generated by $-\Delta+V$ and $\Delta+V^r$ are close, in the sense that for $s\in \R$, $|b|\leq1$  we have that for $0<t<1$ small enough:
\begin{equation}\label{eq:prox_S(t)}
    \Big\|\big(e^{it(-\Delta+V)}-e^{it(-\Delta+V^r)}\big)\phi\Big\|_{X^{-\Delta+V}_{s+\delta,b}}\leeq \|\phi\|_{H^s}.
\end{equation}
In order to prove this we first show that the linear evolutions are close for a given $t$ in  $H^s$ and then in  $X_{s,b}$. Once again (\ref{eq:prox_S(t)}) is equivalent to:
\begin{equation*}
    \Big\|\big(1-e^{-it(-\Delta+V)}e^{it(-\Delta+V^r)}\big)\phi\Big\|_{H^{s+\delta}_xH^b_t}\leeq \|\phi\|_{H^s}.
\end{equation*}
This gain of regularity is crucial as it allows us to treat on the one hand the linear term as being the linear evolution of $\Gg_{-\Delta+V^r}(\omega)$ plus a regular term. We have that thanks to proposition \ref{prop:comutadorlaplmasv} that if $u=e^{it(-\Delta+V^r)}\phi$ for $\phi\in H^s(\T^2) $ that $u$ solves the Cauchy problem
\begin{equation*}
    \begin{cases}
        (i\partial_t-\Delta+V)u =V^au 
        \\
        u(0,x)=\phi.
    \end{cases}
\end{equation*}
Then writing it under its Duhamel form, $u$ satisfies:
\begin{equation*}
    e^{it(-\Delta+V)}\phi-u =i \int_{0}^t e^{i(t-s)(-\Delta+V)}V^a u ds=i \int_{0}^t e^{i(t-s)(-\Delta+V)}V^a e^{is(-\Delta+V^r)}\phi ds
\end{equation*}
yielding
\begin{equation}\label{eq:calculo_norma_xsb_dif}
    \big(1-e^{-it(-\Delta+V)}e^{it(-\Delta+V^r)}\big)\phi=i \int_{0}^t e^{i(-s)(-\Delta+V)}V^a e^{is(-\Delta+V^r)}\phi ds.
\end{equation}
On the other hand if $w=e^{it(-\Delta+V)}\varphi$ we let $\nu=(1+Q)w$, $\nu$ solves:
\begin{equation*}
    \begin{cases}
        (i\partial_t-\Delta+V^r)\nu =R(1+Q)^{-1}\nu 
        \\
        \nu(0,x)=(1+Q)^{-1}\varphi.
    \end{cases}
\end{equation*}
and therefore:
\begin{equation*}
    \nu=e^{it(-\Delta+V^r)}\varphi+ e^{it(-\Delta+V^r)}Q\varphi -i \int_{0}^t e^{i(t-s)(-\Delta+V^r)}R(1+Q)^{-1}\nu ds.
\end{equation*}
From this follows that for $0<t<\tau$, $\tau$ small enough that the two terms on the right can be written as $\widetilde{R}\varphi$ with $\widetilde{R}\in \Ll^{-\delta}$, yielding
\begin{equation*}
    \Big\|\big(e^{it(-\Delta+V)}-e^{it(-\Delta+V^r)}\big)\phi\Big\|_{H^{s+\delta}}\leeq \|\phi\|_{H^s}.
\end{equation*}
Coming back to (\ref{eq:calculo_norma_xsb_dif}) we have:
\begin{equation*}
    -i \int_{0}^t e^{i(-s)(-\Delta+V)}V^a e^{is(-\Delta+V^r)}\phi ds.=-i \int_{0}^t e^{i(-s)(-\Delta+V^r)}V^a e^{is(-\Delta+V^r)}\phi + \widetilde{R}\phi \ ds
\end{equation*}
yielding
\begin{equation}\label{eq:etiqueta_local_1}
    \Big\|\big(1-e^{-it(-\Delta+V)}e^{it(-\Delta+V^r)}\big)\phi\Big\|_{H^{s+\delta}_xH^b_t}\leeq \Big\|\int_{0}^t e^{i(-s)(-\Delta+V^r)}V^a e^{is(-\Delta+V^r)}\phi ds \Big\|_{H^{s+\delta}_xH^b_t}+\|\phi\|_{H^s}.
\end{equation}
Finally, using the same estimates as in the proof of proposition \ref{prop:comutadorlaplmasv} and the same arguments as those used in lemma \ref{lemma:estimacion_1_dualidad} we can bound the first term in the right hand of (\ref{eq:etiqueta_local_1}) for a $\delta_0<\delta/2$, delta being granted in lemma \ref{lem:Partition}.

This means that for $\tau$ small enough and some $s>0$ a $u\in e^{it(-\Delta+V)}\Gg_{-\Delta+V^r}(\omega)+B_{X{s,b}}(0,1)$ may be written as $u\in e^{it(-\Delta+V^r)}\Gg_{-\Delta+V^r}(\omega)+B_{X{s,b}}(0,C(V))$, we have then that we can write 
\begin{equation*}
    e^{it(-\Delta+V^r)}\Gg_{-\Delta+V^r}(\omega)=\sum_{n} \frac{g_n^\omega}{(\lambda_n^r)^{1/2}}e_n^r(x)e^{it\lambda_n^r}.
\end{equation*}
In the following we will write $e_n^r\leftrightarrow e_n$ and $\lambda_n^r\leftrightarrow\lambda_n$ to ease the notation. Take the dyadic decomposition of each of these terms:
\begin{equation*}
u=\sum_{N\in 2^{\N}} \beta_N+\gamma_N=\sum_{N\in 2^{\N}}\sum_{|n|\sim N}\frac{g_n^\omega}{(\lambda_n^r)^{1/2}}e_n^r(x)e^{it\lambda_n^r}+\gamma_N.
\end{equation*}
With $\gamma_N\in B_{X{s,b}}(0,C(V)) $. In the following we work with the dyadic decomposition of $u$ as having a control by a negative power of $N$ suffices to grant summability and thus an estimate on the norm. The dyadic decomposition means that when it comes to estimating contributions we can substitute $\lambda_n^{1/2} \leftrightarrow N$ for $|n|\sim N$. In the following we commit an abuse of notation when writing $u_i=\gamma$ (resp.$\beta$) we mean that $u_i$ is of type $\gamma$ resp.$\beta$. Let us decompose $:u|u|^2:=u|u|^2 -2u \int|u|^2$ into its eigenfunctions, and divide it as:
\begin{align}
&\sum_{\substack{n_2\neq n_1 \\n_2 \neq n_3}} {u}(n_1)\overline{{u}(n_2)}{u}(n_3)e_{n_1}\overline{e_{n_2}}e_{n_3}\label{eq:terminos_decomp_1}\\
+&\sum_{n_2\neq n_1} {u}(n_1)|{u}(n_2)|^2e_{n_1}(|e_{n_2}|^2-1)\label{eq:terminos_decomp_2}\\
+&\sum_{n\in\mathbb{Z}^2}{u}(n)|{u}(n)|^2e_n(|e_n|^2-2)\label{eq:terminos_decomp_3}.
\end{align}
This division is taken so as to be able to make use of the independence of the gaussians. The secon term does not appear in the cases without a potential as we count with a basis of eigenfunctions with constant module $1$. It is this very term that can be used to show that the solution in the sphere deviates from a linear evolution of the initial data on the sphere. 

When it comes to approximating the solution by the truncated equation we can do the following gauge change:
\begin{equation}\label{eq:NLSE_3}
(\ref{eq:NLSE_2})\Leftrightarrow \Big(i\partial_t -\Delta  +V\Big)u^N =-P_N(u^N|u^N|^2-2u^N\int|u^N|^2)+c_N^{(r)}(\omega)
\end{equation}
That allows for the same decomposition in the truncated case. The $u(n)$ depend on $t\in[0,\tau]$ but this division in three terms is independent of time as the $L^2(\T^2)$ is a conserved quantity by the flow. Finally the term (\ref{eq:terminos_decomp_3}) is trivially controlled. Without loss of generality we can assume that $C(V)=1$, as the estimates will hold up to a factor that only depends on $V$, for $\tau$ small enough. 
\subsection{Trilinear expression}
We begin by controlling the trilinear expansion (\ref{eq:terminos_decomp_1}), in order to do this let us write $N^1,N^2,N^3$ the decreasing ordering of $N_i$'s, with $u^1,u^2,u^3$ the corresponding factors. We can assume that $N_1=N^1$, because if we have that $N_2>2(N_1+N_3)$ the resonance of the time frequencies will grant us an $N_2^{-2}$ factor that will grant us the norm control. Thus we are to control the following cases:
\begin{center}
\begin{tabular}{cc}
\begin{tabular}{ |p{1cm}||p{1cm}|p{1cm}|p{1cm}|  }
\hline
\multicolumn{4}{|c|}{Case $N_2=N^2$ and $N_3=N^3$} \\
\hline
& $u_1$ & $u_2$ & $u_3$ \\
\hline
A&$\gamma$&$\gamma$&$\gamma$\\
B&$\gamma$&$\gamma$&$\beta$\\
C&$\gamma$&$\beta$&$\gamma$\\
D&$\gamma$&$\beta$&$\beta$\\
E&$\beta$&$\gamma$&$\gamma$\\
F&$\beta$&$\gamma$&$\beta$\\
G&$\beta$&$\beta$&$\gamma$\\
H&$\beta$&$\beta$&$\beta$\\ 
\hline
\end{tabular}
\quad
\begin{tabular}{ |p{1cm}||p{1cm}|p{1cm}|p{1cm}|  }
\hline
\multicolumn{4}{|c|}{Case $N_3=N^2$ and $N_2=N^3$} \\
\hline
& $u_1$ & $u_2$ & $u_3$ \\
\hline
A'&$\gamma$&$\gamma$&$\gamma$\\
B'&$\gamma$&$\gamma$&$\beta$\\
C'&$\gamma$&$\beta$&$\gamma$\\
D'&$\gamma$&$\beta$&$\beta$\\
E'&$\beta$&$\gamma$&$\gamma$\\
F'&$\beta$&$\gamma$&$\beta$\\
G'&$\beta$&$\beta$&$\gamma$\\
H'&$\beta$&$\beta$&$\beta$\\
\hline
\end{tabular}
\end{tabular} 
\end{center}

\subsubsection{Cases A,B,A' and C'}
We can bound by testing against a function $\|c\|_{X_{-s,b'}}\leq1$. And we get the bound, cf. \cite{bourgain1999global} :
\begin{equation*}
\int|u_1u_2u_3c|\leq (N^2)^\varepsilon\|u^3\|_{X_0,b}\|u^2\|_{X_0,b}\|u^1\|_{X_s,b}.
\end{equation*}
Thus knowing that $N_i^\varepsilon\|u_i\|_{X_0}\leq\|u_i\|_{X_s}$ we have that the cases A,B,A' and C' are solved. We may begin now working on the cases where the highest frequency coresponds to the term in $X_{s,b}$.

\subsubsection{Term standarization}\label{sect:term_standarization}
In order to make use of the counting lemmas we will rewrite the $\gamma$ terms in the same way as the $\beta$ terms but instead of writing them as $e^{it(-\Delta+V^r)}\phi$ we will write them as  $e^{it(-\Delta)}\phi$ as when estimating the contribution of some cases we need the fourier localization. We are able to show that in a sense $u\in X_{s,b}$ then $S(-t)u\approx u_0\in H^s(\T^2)$. We can divide our cases in two, depending on the nature of the highest frequency. Let us begin by supposing that the highest frequency term is $\gamma$, meaning that $u_2$ must be probabilistic else we are in cases, A,B,A' or C'. Let us estimate $(N^1)^s|\int u_1\overline{u_2}u_3\overline{c}|$. We can take $c\sim \Pi_{|n|\leq 4N^1}(c)$ as $|n_1-n_2+n_3|\leq 3N^1$. Let $\mathcal{J}$ be a partition of the set $\{|n|\sim N^1 \}$ in subsets of size $(N^2)^2$, we have by orthogonality that ,letting $\Pi_J$ the projection onto the $-\Delta+V^r$ eigenspaces corresponding to $J\in \mathcal{J}$,
\begin{equation}\label{eq:localtag2}
(N^1)^s\Bigg|\int_{\T^2\times[0,\tau]}u^1u^2u^3c\Bigg|\leeq (N^1)^s \sum_{J\in \mathcal{J}}\int |\Pi_Ju^1| |\Pi_Jc| |u^2||u^3|,
\end{equation}
now through Hölder we get
\begin{equation*}
(\ref{eq:localtag2})\leq \sum_{J\in \mathcal{J}} (N^1)^s\|u^2\|_{L^\infty(\T^2\times[0,\tau])}\| \Pi_Ju^1\|_{L^3(\T^2\times[0,\tau])}\| u^3\|_{L^3(\T^2\times[0,\tau])}\| \Pi_Jc \|_{L^3(\T^2\times[0,\tau])}.
\end{equation*}
We have the further control thanks to (\ref{eq:norma_gaus_infty}), (\ref{coro:stricartz}) and orthogonality :
\begin{align*}
\sum_{J\in \mathcal{J}}\tau^\alpha&(N^1)^s(N^2)^\varepsilon(N^3)^\varepsilon\|u^3\|_{X_{0,1/3}}\|\Pi_Ju^1\|_{X_{0,1/3}}\|\Pi_Jc\|_{X_{0,1/3}}\\ \leq \tau^\alpha &(N^2)^\varepsilon\|u^1\|_{X_{s,1/3}}\|u^3\|_{X_{0,1/3}}\|c\|_{X_{0,1/3}}.
\end{align*}
Forcing once again $b>b'>1/3>0$. Writing the $u_i$ as
\begin{equation*}
\sum_{|n|\sim N }\int_\mathbb{R} d\lambda \widehat{u_n}(\lambda)e^{i(n\cdot x+\lambda t)}
\end{equation*}
this last expression is written as:
\begin{align}\label{eq:Desarrollo_explicito_de_normas_xsb_1}
\tau^\alpha(N^2)^\varepsilon&\Bigg(\sum_{|n_1|\sim N^1} |n_1|^{2s}\int d\lambda_1 \la\lambda_1-|n_1|^2\ra^{2/3}|u_{n_1}(\lambda_1)|^2\Bigg)^{1/2} \times\nonumber \\ &\Bigg(\sum_{|n_3|\sim N^3} \int d\lambda_3 \la\lambda_3-|n_3|^2\ra^{2/3}|u_{n_3}(\lambda_3)|^2\Bigg)^{1/2} \times \\
&\Bigg(\sum_{|n_0|\sim N^1} \int d\lambda_0 \la\lambda_0-|n_0|^2\ra^{2(1/3-b')}|c_{n_0}(\lambda_0)|^2\Bigg)^{1/2}\nonumber.
\end{align}
We have therefore that when $u^1=\gamma$ or $u^3=\gamma$  that the contribution of $u_{n_i}(\lambda_i)$ when $|\lambda_i-|n_i|^2| > (N^2)^{\varepsilon}$ grants us a conclusive result. On the other hand the same applies to $c_{n_0}(\lambda_0)$ for $|\lambda_0-|n_0|^2| > (N^2)^{\varepsilon}$. It is in our interest to take the fourier decomposition of the $\gamma$ terms instead of the decomposition in eigenfunctions as we require a degree of directionality. We can do this as the bourgain spaces defined by $-\Delta$ and $-\Delta+V$ are equivalent. Rewrite for $u=\gamma$:
\begin{align}
u&=\int d\lambda \la\lambda\ra^{-b}e^{i\lambda t}\Big(\la\lambda\ra^{2b}\sum_k|k|^{2s}|u_k(\lambda+|k|^2)|^2\big(\sum_n a_{n}(\lambda)e^{i(n\cdot x+|n|^2t) }\big)^2\Big)^{1/2}
\label{eq:rewrite_gamma}\\&= \int d\lambda \la\lambda\ra^{-b}e^{i\lambda t}\Big(\la\lambda\ra^{2b}\sum_k|k|^{2s}|u_k(\lambda+|k|^2)|^2
\Big)^{1/2}\Big(\sum_n a_{n}(\lambda)e^{i(n\cdot x+|n|^2t)}\Big). \nonumber
\end{align}
Where the $a_n(\lambda)$ are a unitary decomposition of $u$ in the sense:
\[a_n(\lambda)=\frac{\widehat{u_n}(\lambda+|n|^2)}{\Bigg(\sum_k|k|^{2s}|u_k(\lambda+|k|^2)|^2\Bigg)^{1/2}}.\]
Using C-S and the fact that by construction $\sum |n|^{2s}|a_n(\lambda)|^2\leeq1 $ we have the bound for $b>1/2$:
\begin{equation}\label{eq:restrict_log}
\int d\lambda \la\lambda\ra^{-b}\Big(\la\lambda\ra^{2b}\sum_k|\lambda_k|^{s}|u_k(\lambda+\lambda_k)|^2
\Big)^{1/2}\leq C(b).
\end{equation}
This together with the estimations over the $\lambda_i$ being close to the $\lambda_{n_i}$ from (\ref{eq:Desarrollo_explicito_de_normas_xsb_1}), means that up to some factor $C(b)$ we can write the terms $u_i=\gamma$ as  
\begin{equation*}
u_i=e^{i\eta_i t}\sum_na_{n}e^{i(n\cdot x+|n|^2t)}
\end{equation*}
for some $\eta_i\leeq (N^2)^\varepsilon$ with $\sum |n|^{2s}|a_n(\lambda)|^2\leeq1 $. Thus we have to bound :
\begin{equation}\label{eq:norma2}
\tau^\alpha(N^2)^\varepsilon\Bigg(\int d\lambda \la\lambda\ra^{-2b'} \sum_{{|n_0|\leq N^1}}\Bigg|\sum_{\substack{|n_i|\sim N_i \ \ i=1,2,3 \\ \Pp}}a_1(n_1)a_2(n_2)a_3(n_3)\gamma_{n_1,n_2,n_3}^{n_0}(\lambda)\Bigg|^2\Bigg)^{1/2}
\end{equation}
With
\begin{equation*}
\gamma_{n_1,n_2,n_3}^{n_0}(\lambda) = \Pi_\lambda \big( e^{-it(-\Delta+V^r)} \la e^{in_1\cdot x}\overline{e_{n_2}}\widetilde{e}_{n_3}e^{i(\rho+|n_1|^2-\lambda_{n_2}+\widetilde{\lambda_{n_3}}) t},e_{n_0}\ra \big).
\end{equation*}
Where $\widetilde{e}_{n_3}$ is either $e_{n_3}^r$ or $e^{in_3\cdot x}$ depending on weather $u_3$ is $\gamma$ or $\beta$, same for $\widetilde{\lambda_{n_3}}=\lambda_{n_3}^r$ or $|n_3|^2$ 
Here $|\rho|\leeq (N^2)^\varepsilon $. In the following we will loose the $\tau^\alpha$ so as to lighten the writing. The change from $X_{s,b}^{-\Delta+V^r}\leftrightarrow X_{s,b}^{-\Delta}$ decomposition affects the relation between the $n_i$ in the sense that we no longer forcibly have $n_2\neq n_1,n_3$, we introduce the condition $\Pp$ to take this into acount. For example in the case D $\Pp=\{n_2\neq n_3\}$ while the case C sees $\Pp$ doesn't add any constraint. Now let us suppose $u^1$ is random in the same way as above we have the bound:
\begin{align}\label{eq:Desarrollo_explicito_de_normas_xsb_2}
(N^1)^{s+\varepsilon}&\Bigg(\sum_{|n_1|\sim N^1} |n_2|^{2s}\int d\lambda_2 \la\lambda_2-|n_2|^2\ra^{2/3}|u_{n_2}(\lambda_2)|^2\Bigg)^{1/2} \times\nonumber \\ &\Bigg(\sum_{|n_3|\sim N^3} \int d\lambda_3 \la\lambda_3-|n_3|^2\ra^{2/3}|u_{n_3}(\lambda_3)|^2\Bigg)^{1/2} \times \\
&\Bigg(\sum_{|n_0|\sim N^1} \int d\lambda_0 \la\lambda_0-|n_0|^2\ra^{2(1/3-b')}|c_{n_0}(\lambda_0)|^2\Bigg)^{1/2}\nonumber.
\end{align}
With $|\lambda_0-|n_0|^2|<(N^1)^{\alpha s}$ and if $u_i=\gamma$, for $i=1$ or $i=2$, we can further suppose $|\lambda_i-|{n_i}|^2|<(N^1)^{\alpha s}$, for some $\alpha>s/(b''-b')>0$. The previous bound (\ref{eq:norma2}) becomes:
\begin{equation}\label{eq:norma3}
(N^1)^{s+\varepsilon}\Bigg(\int d\lambda \la\lambda\ra^{-2b'} \sum_{{|n_0|\leq N^1}}\Bigg|\sum_{\substack{|n_i|\sim N_i \ \ i=1,2,3 \\ \Pp }}a_1(n_1)a_2(n_2)a_3(n_3)\gamma_{n_1,n_2,n_3}^{n_0}(\lambda)\Bigg|^2\Bigg)^{1/2}.
\end{equation}
This time
\begin{equation*}
\gamma_{n_1,n_2,n_3}^{n_0}(\lambda) = \Pi_\lambda \big( e^{-it(-\Delta+V^r)} \la \tilde{e}_{n_1}\overline{\tilde{e}_{n_2}}\tilde{e}_{n_3}e^{i(\rho+\widetilde{\lambda_{n_1}}-\widetilde{\lambda_{n_2}}+\widetilde{\lambda_{n_3}}) t},e_{n_0}\ra \big).
\end{equation*}
In the first case (\ref{eq:norma2}) thanks to the estimations done by duality in (\ref{eq:Desarrollo_explicito_de_normas_xsb_1}) we can take $|\lambda|\leq (N^2)^{\varepsilon}$ and in (\ref{eq:Desarrollo_explicito_de_normas_xsb_1}) we can take $|\lambda|\leq (N^1)^{\alpha s}$ either way we may take the supremum over the $\lambda$ and then just integrate the $\la\lambda\ra^{-2b'}$ as normal up to $(N^1)^{\alpha s} $ or $ (N^2)^\varepsilon$ which will show conclussive. Thus in the following we look at $\gamma_{n_1,n_2,n_3}^{n_0}(\lambda_0) \to \gamma_{n_1,n_2,n_3}^{n_0}$ for some fixed $\lambda_0$. Let us study the $\gamma_{n_1,n_2,n_3}^{n_0}$ and develop explicitly\footnote{We will suppose in the following that $|n_i| $ is big enough such that we can take $\Pi_{\alpha(n)}e_n=e_n$, the estimates on $\gamma$ hold for the smaller $|n|.$ } $\tilde{e}_{n_1}\overline{\tilde{e}_{n_2}}\tilde{e}_{n_3}$:
\begin{equation*}
    \sum_{\substack{k_i\in\Omega_{\alpha(n_i)}\\i=1,2,3}} \chi_{n_1}^{k_1}\overline{\chi_{n_2}^{k_2}}\chi_{n_3}^{k_3}e^{i(k_1-k_2+k_3)\cdot x} =\sum_{\substack{k_i\\i=1,2,3}} \chi_{n_1}^{k_1}\overline{\chi_{n_2}^{k_2}}\chi_{n_3}^{k_3}\sum_{n_0} \overline{\chi_{n_0}^{k_1-k_2+k_3}}e_{n_0}(x)
\end{equation*}
where the $\chi_{n_i}^{k_i}$ are either $\delta_{n_i=k_i}$, when $\tilde{e}_{n_i}=e^{in_i\cdot x }$, or the $\chi_n^k$ of the Fourier decomposition of the $e_n^r$ when $\tilde{e}_{n_i}=e_{n_i}$. We have used the fact that the $\chi_n^k$ are the coefficients of an orthonormal transformation on $L^2(\T^2)$ from the $\{e_n\}_n\leftrightarrow \{e^{ik\cdot x}\}_k$ and therefore symmetric. We arrive at 
\begin{equation*}
    \gamma_{n_1,n_2,n_3}^{n_0}=\sum_{\substack{k_i\\i=1,2,3}} \chi_{n_1}^{k_1}\overline{\chi_{n_2}^{k_2}}\chi_{n_3}^{k_3}\overline{\chi_{n_0}^{k_1-k_2+k_3}}\delta(\widetilde{\lambda_{n_1}}-\widetilde{\lambda_{n_2}}+\widetilde{\lambda_{n_3}}-\lambda_{n_0}=\rho).
\end{equation*}
Rewrite using the previous notation for $(n_1,n_2,n_3,n_0)$ fixed $k_i=n_i+\mu^1_iu_i^1+\mu_i^2u^2_i$ with $u_i^1=n_i/|n_i|$ and $u_i^2$ a unitary vector that completes an orthonormal base of $\R^2$. Define $n_1-n_2+n_3-n_0=k$ we have that $k=\mu^1_0u_0^1+\mu_0^2u^2_0-(\mu^1_1u_1^1+\mu_1^2u^2_1)+\mu^1_2u_2^1+\mu_2^2u^2_2-(\mu^1_3u_3^1+\mu_3^2u^2_3)$ thus $\la k \ra \le \prod_{i,j}\la\mu_i^j \ra$. Which means that:
\begin{align}\label{eq:demo_gamma_sumable}
   &\Big|\la k\ra^2\sum_{\substack{k_i\\i=1,2,3}} \chi_{n_1}^{k_1}\overline{\chi_{n_2}^{k_2}}\chi_{n_3}^{k_3}\overline{\chi_{n_1-n_2+n_3+k}^{k_1-k_2+k_3}}\Big|^2 \leeq\Bigg(\sum_{\substack{k_i(\mu_i^1,\mu_i^2)\\i=1,2,3}} \prod_{\substack{j=1,2\\i=1,2,3}}\la\mu_i^j\ra^4\big| \chi_{n_1}^{k_1}{\chi_{n_2}^{k_2}}\chi_{n_3}^{k_3}\big|^2\Bigg)\times \nonumber \\ &\la k\ra^4\Bigg(\sum_{\substack{k_i(\mu_i^1,\mu_i^2)\\i=1,2,3}} \prod_{\substack{j=1,2\\i=1,2,3}}\la\mu_i^j\ra^{-4}\big| \chi_{n_0}^{k_1-k_2+k_3}\big|^2\Bigg) \leeq F \sum_{k_0}\big| \chi_{n_0}^{k_0}\big|^2\leq (N^2)^\varepsilon .
\end{align}
The first inequality coming from Cauchy Schwartz, for the second inequality we use the second control in lemma \ref{lem:control_e_f} to bound the first sum by a constant. For the second factor we have to distinguish the cases. If we are in the cases C,B',D,D' $\chi_{n_1}^{k_1}=\delta_{n_1=k_1}$ and thus for a fixed $k_0$ there is at most $(N^3)^{\rho}$ duos of $(k_2,k_3)$ such that $k_0=n_1-k_2+k_3$ as well as the previously stated $\la k \ra \prod_{i,j}\la\mu_i^j \ra^{-1}\leeq 1$, meaning that we take $F=(N^3)^{\rho}$. Furthermore in the cases C,C' $\chi_{n_3}^{k_3}=\delta_{n_3=k_3}$ grants $F=1$. On the other hand for the rest it suffices to take $F=(N^2)^\rho$. Finally it we can apply the first control from lemma \ref{lem:control_e_f} to get the last inequality. This yields:
\begin{equation*}
    | \gamma_{n_1,n_2,n_3}^{n_1-n_2+n_3+k}|^2\le F \delta(\widetilde{\lambda_{n_1}}-\widetilde{\lambda_{n_2}}+\widetilde{\lambda_{n_3}}-\lambda_{n_1-n_2+n_3+k}=\rho)\la k\ra ^{-4}.
\end{equation*}
Now we introduce the counting lemmas that will allow us to gain regularity by limiting the number of resonant cases. but first remark that $\widetilde{\lambda_{n_1}}-\widetilde{\lambda_{n_2}}+\widetilde{\lambda_{n_3}}-\lambda_{n_0}=\rho \implies |n_1|^2-|n_2|^2+|n_3|^2-|n_0|^2=\rho+C$ with C being uniformly controlled by a constant depending on $V$. and thus we can control, up to a constant, the set of resonances by their integer version. We thus want to adapt the lattice counting lemmas given in \cite{2D_Defocusing}. Let for $\rho\in \R^+$, $n_0\in\Z^2$, and $k\in \Z^2$, $\Lambda_{n_0,k,\rho}(N_1,N_2,N_3):=$
\begin{align*}
&\#\left\{(n_1,n_2,n_3)\in (\Z^2)^3 \ | \ \substack{n_2\neq n_1 \\ \ n_2\neq n_3 }, \substack{|n_i|\sim N_i \\ i=1,2,3}, n_1-n_2+n_3=n_0; \  |n_1|^2-|n_2|^2+|n_3|^2-|n_0+k|^2 =\rho \right\}\\
&=\#\left\{(n_1,n_2,n_3)\in (\Z^2)^3 \ | \ \substack{n_2\neq n_1 \\ \ n_2\neq n_3 }, \substack{|n_i|\sim N_i \\ i=1,2,3}, n_1-n_2+n_3=n_0; \  |n_1|^2-|n_2|^2+|n_3|^2-|n_0|^2 =\rho' \right\}.
\end{align*}
With $\rho'=\rho+|k|^2+2\la n_0,k\ra$. This is directly bounded in \cite{2D_Defocusing} by
\begin{lem}\label{lemma:lema_contar_1}
\begin{equation}
\Lambda_{k,n_0,\rho}(N_1,N_2,N_3)\ll C(\varepsilon)\min \left\{(N_2)^2(N_1\wedge N_3)^\varepsilon,N^2N^3(N^3)^\varepsilon\right\}.
\end{equation}
Where the upper script means that $N^1\geq N^2\geq N^3$ is the ordered set of $N_1,N_2$ and $N_3$.
\end{lem}
We need to make account for how many elements there are for fixed $n_i$ this is why we introduce the following lemma.
\begin{lem}\label{lemma:lema_contar_2}
Let for some $N_1,N_2,N_3$ taken dyadic we will note $S^{i_1,i_2,\dots}_{\rho,j_1,{j_2},\dots}(k)$ the set
\[\left\{(n_1,n_2,n_3)\in (\Z^2)^3 \ | \ \substack{n_2\neq n_1 \\ \ n_2\neq n_3 }, \substack{|n_i|\sim N_i \\ i=1,2,3}; \  |n_1|^2-|n_2|^2+|n_3|^2-|n_1-n_2+n_3+k|^2 =\rho \right\},\]
Where the $n_{j_1},n_{j_2},\dots$ are fixed and the $n_{i_1},n_{i_2},\dots$ are the ones we are counting we have that:
\begin{enumerate}[i)]
\item$\#S_{\rho, n_1}^{n_2,n_3}(k)\ll  (N_3)^2(N_2)^\varepsilon$
\item$\#S_{\rho, n_2}^{n_1,n_3}(k)\ll N_1 N_3(N_1\wedge N_3)^\varepsilon$
\item $\#S_{\rho, n_1, n_2}^{n_3}(k)< N_3$
\item $\#S_{\rho, n_1, n_3}^{n_2}(k)< (N_2)^\varepsilon$.
\end{enumerate}
\end{lem}
\begin{proof}
Begin noting for any $(n_1,n_2,n_3)$ $n_0=n_1-n_2+n_3$, for a fixed $k,\rho$ we have to find the amount of $n_i$ that satisfy:
    \begin{equation}\label{eq:prod_a_calcular_numeros}
        \la n_2-n_1,n_2-n_3 \ra+\la k,n_0\ra=\rho'
    \end{equation}
where $\rho'=(|k|^2+\rho)/2$, we will use different formulations of this equation to obtain our estimates.

For $i)$ consider rewriting (\ref{eq:prod_a_calcular_numeros}) as
\[\la n_2-n_1-k,n_2-n_3 \ra+\la k,n_1\ra=\rho'\]
which yields for $\rho''=\rho'-\la k,n_1\ra$
\[\Bigg| n_2-\frac{n_1+k+n_3}{2}\Bigg|-\Bigg|\frac{n_1+k-n_3}{2}\Bigg|^2=\rho'\]
for a fixed $n_3$ the amount of $n_2$ that satisfy this is bounded by the number of elements in $\Z^2$ that are in a circle of radius smaller than $N^2$, meaning that we get the bound $(N_3)^2(N_2)^\varepsilon$. 

As for $ii)$ we fix $n_2$ and define for some $r>0$ and $a,b$ relatively prime $n_2-n_3-k=r(a,b)$ we have that the equation becomes
\begin{equation*}
    \la n_1,(a,b)\ra= \la n_2',(a,b) \ra +\rho''
\end{equation*}
with $n_2'=n_2-k$ and $\rho''=(-\rho'-\la n_2',k\ra)/r$. If $(a,b)=(0,0)$ we can estimate by $(N_1)$ as we will get that $n_3$ is uniquely decided by $n_2$. Therefore suppose that $b\neq0$. First if $a=0$ we get $N_3$ choices for $n_3$ and then $\la n_2-n_1-k,n_2-n_3 \ra+\la k,n_1\ra=\rho'$ forces only $N_1$ choices for the $n_1$. Finally suppose that $a\neq 0$, and take $n_3>|a|\vee |b|$ we have that the number of $n_3$ so that $n_3>|a|\vee |b|$ for some $A<|a|\le2A$ and $B<|b|\le2B$ such that $n_2'-n_3=r(a,b)$ is at most $N_3/(A\vee B)$. Thus by symmetry in $n_1$ is estimated in the same way by $N_1/(A\vee B)$. We arrive adding over the $A,B$ the control $\sum_{A,B}\sum_{|a|,||b|}N_1/(A\vee B)N_3/(A\vee B)\leeq N_1(N_3)^{1+\varepsilon}$. On the other hand if w have that $n_3<|a|\vee |b|$ we take the fact that there are $N_1/(|a|\vee|b|)+1$ solutions in $n_1$ and therefore there is at most $N_1/N_3$ solutions for fixed $n_3$ and we deduce directly $N_1N_3$ in total. 

For $iii)$ we have that $n_1,n_2$ are fixed and we have to calculate the number of solutions in $n_3$ to $\la n_3,k'\ra=\rho'''$ which is bounded by a dimensionality argument by $N_3$. 

Finally $iv)$ is equivalent to searching the number of $n_2$ solution to $|n_2|^2+\la n_2,k'\ra=\rho'''$ which is bounded by $(N_2)^\varepsilon$.
\end{proof}
\begin{rem}
    It is also important to realize that we can always exchange the role of the $n_1\leftrightarrow n_3$ and so. For example $\#S_{\rho, n_1}^{n_2,n_3}(k)\ll  (N_3)^2(N_2)^\varepsilon \implies \#S_{\rho, n_3}^{n_1,n_2}(k)\ll  (N_2)^2(N_1)^\varepsilon$ this type of exchanges will be very useful later.
\end{rem}
We are now equipped to treat all remaining cases. In the following we will replace $|n_i|\leftrightarrow N_i$ when it comes to estimating as the difference is a factor smaller than $2$. This will show to be mostly a repetition of previous arguments, but we write them down as the introduction of the different forms of $\gamma_{n_1,n_2,n_3}^{n_0}$ will show tricky at times.
\subsubsection{Cases H and H'} We begin by the purely probabilistic cases H and H'. We are in the case (\ref{eq:norma3}) with the three $a_i(n_i)=\frac{g_{n_i}^\omega }{N_i}$. If we first take the terms where $n_1,n_2,n_3$ are different two by two, we have that thanks to subgaussianity and our gaussians being independent the control\footnote{of the squared term}:
\begin{align*}
\frac{(N^1)^{2s+\varepsilon}}{(N_1N_2N_3)^2}&\sum_{{|n_0|\leq N^1}}\Bigg|\sum_{\substack{n_2\neq n_1 \ \ n_2\neq n_3 \\ n_1\neq n_3}} {g_{n_1}^\omega }{g_{n_2}^\omega }{g_{n_3}^\omega }\gamma_{n_1,n_2,n_3}^{n_0}\Bigg|^2\leeq \\ \frac{(N^1)^{2s+\varepsilon}}{(N_1N_2N_3)^2}&\sum_{{k}}\sum_{\substack{n_2\neq n_1 \ \ n_2\neq n_3 \\ n_1\neq n_3}}|\gamma_{n_1,n_2,n_3}^{n_1-n_2+n_3+k}|^2\leeq \\
 \frac{(N^1)^{2s+\varepsilon}}{(N_1N_2N_3)^2}&\sum_{{k}}\sum_{\substack{n_2\neq n_1 \ \ n_2\neq n_3 \\ n_1\neq n_3}}\delta(\lambda_{n_1}-\lambda_{n_2}+\lambda_{n_3}-\lambda_{n_1-n_2+n_3+k}=\rho)({N^2})^\rho \la k\ra ^{-4}
\end{align*}
for $\omega$ outside of a set of size $e^{-C(N_1)^\varepsilon/\tau^{\alpha}}$. As seen before the fact that there is a universal constant such that $|\lambda_n-|n||\le C(V)$ allows us to use lemma's \ref{lemma:lema_contar_2} $iii)$ with $n_1 \leftrightarrow n_3$:
\begin{equation*}
\frac{(N^1)^{2s+\varepsilon}({N^2})^\rho}{(N_1N_2N_3)^2}\sum_{{k}}\sum_{n_2,n_3}\#S_{\rho, n_2, n_3}^{n_1}(k) \la k\ra ^{-4}\le \frac{(N^1)^{1+2s+\varepsilon}({N^2})^\rho}{(N_1N_2N_3)^2}\sum_{{k}}\sum_{n_2,n_3}\la k\ra ^{-4}\leeq (N^1)^{2s+\rho+\varepsilon-1}
\end{equation*}
Which is conclusive for $s<(1-\rho)/2$. On the other hand if $n_1=n_3$ we get
\begin{align*}
\frac{(N^1)^{2s+\varepsilon}}{(N_1)^4(N_2)^2}&\sum_{{|n_0|\leq N^1}}\Bigg|\sum_{\substack{n_2\neq n_1}} {(g_{n_1}^\omega)^2}{g_{n_2}^\omega }\gamma_{n_1,n_2,n_1}^{n_0}\Bigg|^2\leeq \\ \frac{(N^1)^{2s+\varepsilon}}{(N_1)^4(N_2)^2}&\sum_{{k}}\sum_{\substack{n_2\neq n_1}}|\gamma_{n_1,n_2,n_1}^{2n_1-n_2+k}|^2\leeq \frac{(N^1)^{2s+\varepsilon}({N^2})^{\rho+\varepsilon}}{(N_1)^4}
\end{align*}
\subsubsection{Case C} We now look at the estimate (\ref{eq:norma2}) and remark that $N_2=N^2$ and $N_3=N^3$. Using Cauchy-Schwarz we can control the square of sum by:
\begin{align*}\label{eq:TerminoCcontrol}
&(N^2)^{\varepsilon-2}\sum_{{n_0}}\Bigg|\sum_{\substack{|n_i|\sim N_i \ \ i=1,2,3 }}a_1(n_1)g_{n_2}^\omega a_3(n_3)\gamma_{n_1,n_2,n_3}^{n_0}\Bigg|^2\leq\\
&(N^2)^{\varepsilon-2}\sum_{{n_0}}\Bigg(\sum_{n_1,n_2,n_3}\delta(\lambda_{n_1}-\lambda_{n_2}+\lambda_{n_3}-\lambda_{n_0}=\rho)\Bigg)\Bigg(\sum_{n_1,n_2,n_3}|a_1(n_1)|^2|a_3(n_3)|^2|\gamma_{n_1,n_2,n_3}^{n_0}|^2\Bigg).
\end{align*}
Where we have taken $\omega$ outside of a set of size $e^{-C(N_1)^\varepsilon/\tau^{\alpha}}$. The first factor can be bounded for a fixed $n_0$ thanks to lemma \ref{lemma:lema_contar_1} by $N_2(N_3)^{1+\varepsilon}$, yielding:
\begin{align*}
&(N_2)^{\varepsilon-1}(N_3)^{1+\varepsilon}\sum_{{k}}\sum_{n_1,n_2,n_3}|a_1(n_1)|^2|a_3(n_3)|^2|\gamma_{n_1,n_2,n_3}^{n_1-n2+n_3+k}|^2\leeq\\ 
&(N_2)^{\varepsilon-1}(N_3)^{1+\varepsilon-2s}\sum_{k,n_2}\#S_{\rho, n_1, n_3}^{n_2}(k)\la k\ra^{-4}  \leeq ((N^2)^{\varepsilon-1/2}(N^3)^{1/2+\varepsilon-s})
\end{align*}
Therefore we can control the contribution of term C by $ (N^2)^{\varepsilon-s} $ for all $0<s<1/2$. 

\subsubsection{Case B'} This time we use (\ref{eq:norma2}) and we control using Cauchy-Schwarz on the $n_2$ variable.
\begin{align*}
&(N_3)^{\varepsilon-2}\sum_{{n_0}}\Bigg|\sum_{\substack{|n_i|\sim N_i \ \ i=1,2,3 }}a_1(n_1)a_2(n_2)g_{n_3}^\omega \gamma_{n_1,n_2,n_3}^{n_0}\Bigg|^2\leq\\
&(N_2)^{-2s}(N_3)^{\varepsilon-2}\sum_{{n_0,n_2}}\Bigg|\sum_{\substack{|n_i|\sim N_i \ \ i=1,2,3 \\}}a_1(n_1)g_{n_3}^\omega \gamma_{n_1,n_2,n_3}^{n_0}\Bigg|^2.
\end{align*}
Knowing that $a_3$ is gaussian we have, thanks to hypercontractivity, up to restriction in $\omega$:
\begin{align*}
&(N_2)^{-2s}(N_3)^{\varepsilon-2}\sum_{{n_0,n_2,n_3}}\Bigg|\sum_{\substack{|n_i|\sim N_i \ \ i=1,2,3 \\ n_2\neq n_1 \ \ n_2\neq n_3}}a_1(n_1)\gamma_{n_1,n_2,n_3}^{n_0}\Bigg|^2\le \\ &(N_2)^{-2s}(N_3)^{\varepsilon-2}\sum_{n_0,n_2,n_3}\Bigg(\sum_{n_1}|a_1(n_1)|^2|\gamma_{n_1,n_2,n_3}^{n_0}|^{2-\alpha}\Bigg)\Bigg(\sum_{n_1'}|\gamma_{n_1',n_2,n_3}^{n_0}|^\alpha\Bigg).
\end{align*}
If we take $1/2<\alpha$ we have that for a fixed $n_0,n_2,n_3$ that $|\gamma_{n_1',n_2,n_3}^{n_0}|\leeq \delta(\lambda_{n_1}-\lambda_{n_2}+\lambda_{n_3}-\lambda_{n_0}=\rho)\la n_1-n_2+n_3-n_0 \ra^{-4\alpha} $ is summable in $n_1'$ by dimensionality. If we take on top of that $1/2<\alpha<1$ we are left with
\begin{align*}
&(N_2)^{-2s}(N_3)^{\varepsilon-2}\sum_{n_1}|a_1(n_1)|^2\Bigg(\sum_{k,n_2,n_3}|\gamma_{n_1,n_2,n_3}^{n_1-n_2+n_3+k}|^{2-\alpha}\Bigg)\le \\ 
&(N_2)^{-2s}(N_3)^{\varepsilon-2}\sum_{n_1}|a_1(n_1)|^2\Bigg(\sum_{k}\#S_{\rho, n_1}^{n_2,n_3}(k)\la k\ra^{-2(2-\alpha)}\Bigg)\le N_2^{1-2s+\varepsilon}(N_3)^{\varepsilon-1} .
\end{align*}
which can be controlled by $(N_3)^{\varepsilon-2s}$ for $s<1/2$.
\subsubsection{Cases E and E'} It is enough to control case E as it is the same for for E' with the change $N_2 \longleftrightarrow N_3$. In the estimation of (\ref{eq:norma3}) we first use the previous method applied in case B' with the role exchange $n_1 \leftrightarrow n_3$:
\begin{align*}
&(N_1)^{\varepsilon+2s-2}\sum_{{n_0}}\Bigg|\sum_{\substack{|n_i|\sim N_i \ \ i=1,2,3 \\ n_2\neq n_1 \ \ n_2\neq n_3}}g_{n_1}^\omega a_2(n_2)a_3(n_3)\gamma_{n_1,n_2,n_3}^{n_0}\Bigg|^2\leq\\
&(N_1)^{\varepsilon+2s-2}(N_2)^{-2s}\sum_{{n_0,n_1,n_2}}\Bigg(\sum_{n_3}|a_3(n_3)|^2|\gamma_{n_1,n_2,n_3}^{n_0}|^{2-\alpha}\Bigg)\Bigg(\sum_{n_3'}|\gamma_{n_1,n_2,n_3'}^{n_0}|^\alpha\Bigg)\le\\
&(N_1)^{\varepsilon+2s-2}(N_2)^{-2s}(N_2)^{1+\varepsilon}N_1(N_3)^{-2s}\le (N_1)^\varepsilon\Bigg(\frac{N_1}{N_2}\Bigg)^{2s-1}(N_3)^{-2s},
\end{align*}
as always up to restriction in $\omega$. This estimate is not conclusive for all $s<1/2$ so we now get another estimate that we can use to get the desired bound through interpolation. We begin by doing C-S first over $n_3$ then over $n_2$ and we get the estimate:
\begin{equation*}
(N_1)^{\varepsilon+2s-2}(N_3)^{-2s}\sum_{{n_0,n_3}}\Bigg|\sum_{\substack{|n_i|\sim N_i \ \ i=1,2,3 \\ n_2\neq n_1 \ \ n_2\neq n_3}}g_{n_1}^\omega a_2(n_2)\gamma_{n_1,n_2,n_3}^{n_0}\Bigg|^2
\end{equation*}
For a fixed $n_3$ we can see this as the norm of the image of bounded vector of norm $(N_2)^{-s}$by the following random matrix :
\begin{equation*}
\cM_{n_0,n_2}:=\sum_{\substack{|n_1|\sim N_1 \\ n_2\neq n_1 \ \ n_2\neq n_3}} g_{n_1}^\omega \gamma_{n_1,n_2,n_3}^{n_0}
\end{equation*}
Now, adding over the $n_3$ we get the estimation of by the operator norm:
\begin{equation*}
(N_1)^{\varepsilon+2s-2}(N_2)^{-2s}(N_3)^{2-2s}\|\cM^*\cM\|
\end{equation*}
We can estimate the norm of $\cM^*\cM$ by:
\begin{equation}\label{eq:norma_matrx}
\|\cM^*\cM\|\leq \max_{n_0}\Bigg(\sum_{n_2}|\cM_{n_0,n_2}|^2\Bigg)+\sum_{n_0\neq n_0'}\Bigg|\sum_{n_2}\cM_{n_0,n_2}\overline{\cM_{n_0',n_2}}\Bigg|^2.
\end{equation}
For some fixed $n_3,n$, the first term is controlled, up to restriction in $\omega$ by:
\begin{equation*}
   (N_1)^{\varepsilon}\sum_{n_1,n_2}|\gamma_{n_1,n_2,n_3}^{n_0}|^2 \leeq \sum_k \#S_{-\rho, n_1', n_2'}^{n_3'}(k)\la k \ra^{-4}< (N_1)^{\varepsilon}N_2
\end{equation*}
As said before we have exchanged the $n_1'\leftrightarrow n_0,n_2' \leftrightarrow n_3,n_3'\leftrightarrow n_2$. Now for the second term (\ref{eq:norma_matrx}) the control is 
\begin{equation*}
\sum_{n_0\neq n_0'}\Bigg|\sum_{\substack{n_1,n_1'\\n_2}}{g_{n_1}^\omega }\gamma_{n_1,n_2,n_3}^{n_0}\overline{{g_{n_1'}^\omega }\gamma_{n_1',n_2,n_3}^{n_0'}}\Bigg|^2
\end{equation*}
The pairs $\{n_1,n_1'\}$ repeat only twice so we have up to restriction in $\omega$ the following control:
\begin{equation*}
    \sum_{n_0\neq n_0'}\sum_{\substack{n_1,n_1'\\n_2}}|\gamma_{n_1,n_2,n_3}^{n_0}|^2|\gamma_{n_1',n_2,n_3}^{n_0'}|^2=\sum_{k\neq k'}\sum_{\substack{n_1,n_1'\\n_2}}|\gamma_{n_1,n_2,n_3}^{n_1-n_2+n_3+k}|^2|\gamma_{n_1',n_2,n_3}^{n_1'-n_2+n_3+k'}|^2
\end{equation*}
with always $n_1-n_2+n_3+k\neq n_1'-n_2+n_3+k'\implies n_1'+k'\neq n_1+k$. The number of pairs $(n_1,n_2)$ can be bounded, exchanging the role of $n,n_2$ and $n_1,n_3$, by $\#S_{\rho, n_3}^{n_1,n_2}(k)\ll C(\varepsilon)N_1 N_2(N_1\wedge N_2)^\varepsilon$ and on the other hand the choices over the $n'$ are controlled by an argument of dimensionality by $N_1$. Which yields the control:
\begin{equation*}
(N^1)^{\varepsilon+2s}(N_2)^{-4}(N_3)^{1-s}((N_1)^{-2+\varepsilon}(N_2)^{1+\varepsilon})^{1/2}\leq (N_1)^{\varepsilon-1/2}\Bigg(\frac{N_1}{N_2}\Bigg)^{2s-1/2}(N_3)^{2-2s}.
\end{equation*}
Interpolate both controls to get the contribution of this term bounded by $(N_1)^{\varepsilon-s/2}$.
\subsubsection{Case D} We have to control the square root of:
\begin{equation}\label{eq:casoD}
(N^2)^{\varepsilon-2}(N^3)^{-2}\sum_{{n_0}}\Bigg|\sum_{\substack{|n_i|\sim N_i \ \ i=1,2,3 \\  n_2\neq n_3}}a_1(n_1)\overline{g_{n_2}^\omega }
g_{n_3} ^\omega \gamma_{n_1,n_2,n_3}^{n_0}\Bigg|^2\leq.
\end{equation}
In a similar manner as what we did with cases E and E' we can control de random part by its matrix operator norm, that is, we have that for $n_0$ in a subset of size $(N_2)^2$ and $n_1$ in the same subset:
\begin{equation*}
(\ref{eq:casoD})^{1/2}\leq(N_2)^{\varepsilon-1}(N_3)^{-1}\|\cM^*\cM\|^{1/2},
\end{equation*}
where the matrix $\cM$ is defined as:
\begin{equation*}
\cM_{n,n_1}=\sum_{\substack{n_{2,3}\sim N_{2,3} \\ \ \ n_2\neq n_3}} \overline{g_{n_2}^\omega }g_{n_3}^\omega \gamma_{n_1,n_2,n_3}^{n_0}.
\end{equation*}
We can estimate the norm of $\cM^*\cM$ as before by:
\begin{equation*}
\|\cM^*\cM\|\leq \max_n\Bigg(\sum_{n_1\sim N^1}|\cM_{n,n_1}|^2\Bigg)+\Bigg(\sum_{n\neq n'}\Bigg|\sum_{n_1}\cM_{n,n_1}\overline{\cM_{n',n_1}}\Bigg|^2\Bigg)^{1/2}.
\end{equation*}
For the first term in the sum we can restrict in $\omega$ and have for some fixed $n$ :
\begin{equation*}
\sum_{n_1}\sum_{\substack{n_{2,3}\sim N_{2,3} \\ n_2\neq n_3}} |\gamma_{n_1,n_2,n_3}^{n_0}|^2\leeq (N_3)^2(N_2)^{1+\rho+\varepsilon}.
\end{equation*}
This comes from using lemma's \ref{lemma:lema_contar_2} third point exchanging the roles of $n_1,n_3\leftrightarrow n_0,n_2$ and fixing $n_3$. Let us look at the other terms, we have:
\begin{equation*}
\sum_{n_0\neq n_0'}\Big|\sum_{\substack{n_1,n_2,n_2'\\ n_3,n_3'}}\overline{g_{n_2}^\omega }g_{n_3}^\omega \overline{g_{n_3'}^\omega }g_{n_2'}^\omega \gamma_{n_1,n_2,n_3}^{n_0}\overline{\gamma_{n_1,n_2',n_3'}^{n_0'}}\Big|^2.
\end{equation*}
\begin{enumerate}
\item Suppose first that $n_2,n'_2,n_3,n'_3$ are different two by two. We then have the estimate up to restriction in $\omega$:
\begin{equation*}
\sum_{n_0\neq n_0'}\sum_{\substack{n_1,n_2,n_2'\\ n_3,n_3'}}\Big|\gamma_{n_1,n_2,n_3}^{n_0}{\gamma_{n_1,n_2',n_3'}^{n_0'}}\Big|^2.
\end{equation*}
Fixing $n_1$ and applying the $i)$ from lemma \ref{lemma:lema_contar_2} and after adding over the $n_1$ we get $(N_2)^{2+\rho+\varepsilon}(N_3)^4$. 
\item  On the other hand he case $n_2 = n_2'$ but $n_3 \neq n_3'$. We get then:
\begin{align*}
&\sum_{n_0\neq n_0'} \Big|\sum_{\substack{n_1,n_2,\\ n_3,n_3'}}|g_{n_2}^\omega |^2g_{n_3}^\omega \overline{g_{n_3'}^\omega }\gamma_{n_1,n_2,n_3}^{n_0}\overline{\gamma_{n_1,n_2,n_3'}^{n_0'}}\Big|^2\leeq \\ &(N_3)^{\varepsilon}\sum_{\substack{n_0\neq n_0'\\ n_3,n_3'}}\Big(\sum_{\substack{n_1,n_2}}\Big|\gamma_{n_1,n_2,n_3}^{n_0}{\gamma_{n_1,n_2,n_3'}^{n_0'}}\Big|^{2-\alpha}\Big)\Big(\sum_{\substack{n_1,n_2}}\Big|\gamma_{n_1,n_2,n_3}^{n_0}{\gamma_{n_1,n_2,n_3'}^{n_0'}}\Big|^{\alpha}\Big)
\end{align*}
For some $1/2<\alpha<1$ as before. Let us examine for fixed $n_0,n_0',n_3,n_3'$ the second sum is controlled by, up to a $(N_2)^{\alpha\rho}$ factor
\begin{align*}
&\sum_{\substack{n_1,n_2}} \delta\bigg({\substack{\lambda_{n_1}-\lambda_{n_2}+\lambda_{n_3}-\lambda_{n_0}=\rho\\\lambda_{n_1}-\lambda_{n_2}+\lambda_{n_3'}-\lambda_{n_0'}=\rho}}\bigg) \la n_1-n_2+n_3-n_0 \ra^{-4\alpha}\la n_1-n_2+n_3'-n_0' \ra^{-4\alpha}=\\
& \sum_{\substack{n_2,k}} \delta \bigg({\substack{\lambda_{n_2+k}-\lambda_{n_2}+\lambda_{n_3}-\lambda_{n_0}=\rho\\\lambda_{n_2+k}-\lambda_{n_2}+\lambda_{n_3'}-\lambda_{n_0'}=\rho}}\bigg)\la k+n_3-n_0 \ra^{-4\alpha}\la k+n_3'-n_0' \ra^{-4\alpha}\leeq N_2.
\end{align*}
The last inequality is given by the $iii)$ term in lemma \ref{lemma:lema_contar_2}. On the other hand we can see that 
\begin{equation*}
    (N_3)^{\varepsilon}\sum_{\substack{n_0\neq n_0'\\ n_3,n_3'}}\sum_{\substack{n_1,n_2}}\Big|\gamma_{n_1,n_2,n_3}^{n_0}{\gamma_{n_1,n_2,n_3'}^{n_0'}}\Big|^{2-\alpha}\leeq (N_2)^{2+\rho+\varepsilon}(N_3)^3
\end{equation*}
coming from applying $i)$ and then $iii)$, this yields finally $(N_2)^{3+\rho+\varepsilon}(N_3)^3$.
\item Now take the case $n_2\neq n_2'$ but $n_3=n_3'$. We have as before for an $1/2<\alpha<1$ that:
\begin{align*}
&\sum_{n_0\neq n_0'} \Big|\sum_{\substack{n_1,n_2,\\ n_2',n_3}}g_{n_2}^\omega \overline{g_{n_2'}^\omega }|g_{n_3}^\omega |^2\gamma_{n_1,n_2,n_3}^{n_0}\overline{\gamma_{n_1,n_2',n_3}^{n_0'}}\Big|^2\leeq \\ &(N_3)^{\varepsilon}\sum_{\substack{n_0\neq n_0'\\ n_2,n_2'}}\Big(\sum_{\substack{n_1,n_3}}\Big|\gamma_{n_1,n_2,n_3}^{n_0}{\gamma_{n_1,n_2',n_3}^{n_0'}}\Big|^{2-\alpha}\Big)\Big(\sum_{\substack{n_1,n_3}}\Big|\gamma_{n_1,n_2',n_3}^{n_0}{\gamma_{n_1,n_2,n_3}^{n_0'}}\Big|^{\alpha}\Big).
\end{align*}
On the one hand the second sum (for fixed ${n_0, n_0', n_2,n_2'}$) is bounded by $(N_3)^{\varepsilon}$ through $iv)$, on the other hand the rest is bounded by $(N_2)^{2+\rho+\varepsilon}(N_3)^2$, granting a final contribution of $(N_2)^{2+\rho+\varepsilon}(N_3)^{2+\varepsilon}$.
\item Next up is the case $n_2\neq n_3'$ but $n_3=n_2'$. We can achieve the control:
\begin{align*}
&\sum_{n_0\neq n_0'} \Big|\sum_{\substack{n_1,n_2,\\ n_2',n_3}}|g_{n_2}^\omega |^2\overline{g_{n_2'}^\omega }g_{n_3}^\omega \gamma_{n_1,n_2,n_3}^{n_0}\overline{\gamma_{n_1,n_2,n_2'}^{n_0'}}\Big|^2\leeq \\ &(N_2)^{\varepsilon}\sum_{\substack{n_0\neq n_0'\\ n_3,n_2'}}\Big(\sum_{\substack{n_1,n_2}}\Big|\gamma_{n_1,n_2,n_3}^{n_0}{\gamma_{n_1,n_2,n_2'}^{n_0'}}\Big|^{2-\alpha}\Big)\Big(\sum_{\substack{n_1,n_2}}\Big|\gamma_{n_1,n_2,n_3}^{n_0}{\gamma_{n_1,n_2,n_2'}^{n_0'}}\Big|^{\alpha}\Big)\leeq (N_2)^{3+\varepsilon}(N_3)^2.
\end{align*}
Where we can control the second sum by $N_2$ thanks to $iii)$ and the rest is controlled thanks to $i)$ on the full set for fixed $n_1$ and then $iv)$. The symmetric case $n_2= n_3'$ but $n_3\neq n_2'$ is treated exactly the same.
\item Finally we look at the contribution of $n_2=n_3'$ and $n_3=n_2'$. Here we are to control:
\begin{align*}
&\sum_{n_0\neq n_0'} \Big|\sum_{\substack{n_1,n_2,\\n_3}}|g_{n_2}^\omega |^2|g_{n_3}^\omega |^2\gamma_{n_1,n_2,n_3}^{n_0}\overline{\gamma_{n_1,n_3,n_2}^{n_0'}}\Big|^2\leeq \\ &(N_2)^{\varepsilon}\sum_{\substack{n_0\neq n_0'}}\Big(\sum_{\substack{n_1,n_2,n_3}}\Big|\gamma_{n_1,n_2,n_3}^{n_0}{\gamma_{n_1,n_3,n_2}^{n_0'}}\Big|^{2-\alpha}\Big)\Big(\sum_{\substack{n_1,n_2,n_3}}\Big|\gamma_{n_1,n_2,n_3}^{n_0}{\gamma_{n_1,n_3,n_2}^{n_0'}}\Big|^{\alpha}\Big).
\end{align*}
If we look at the second factor we can bound it by, up to a $(N_2)^{\rho}$ factor
\begin{equation*}
\sum_{\substack{k,n_2,n_3}} \delta\bigg({\substack{\lambda_{n_2-n_3+n_0+k}-\lambda_{n_2}+\lambda_{n_3}-\lambda_{n_0}=\rho\\\lambda_{n_2-n_3+n_0'+k_0+k}-\lambda_{n_2}+\lambda_{n_3}-\lambda_{n_0'}=\rho}}\bigg)\la k\ra^{-4\alpha}\la k+k_0\ra^{-4\alpha}\leeq N_3
\end{equation*}
where $k_0=n_0'-n_0\neq 0$, and the $N_3$ comes directly from estimation $iii)$ and the fact that $n_0\neq n_0'$. On the other hand the first sum can be easily bounded by a dimensionality argument on the $n_2$ for a fixed $n_3$ by $N_2$ and then after adding over the $n_3$ by $(N_3)^2(N_2)^{1+\rho}$, thus the contribution of this terms is $(N_2)^{2+\varepsilon+\rho}(N_3)^3$.
\end{enumerate}
Collecting all the previous estimates and the factor $(N_2N_3)^{-1}$ we get that up to restriction in $\omega$ $\|\cM^*\cM\|\leeq (N_2)^{-1/2+\rho+\varepsilon}$, and thus the contribution of case D bounded by $(N_2)^{-1/4+\rho+\varepsilon}$. 
\subsubsection{Case D'} We can adapt the previous approach but with $N_2\le N_3$ in this case, if $N_2\sim N_3$ in the sense that $N_2>(N_3)^{9/10}$ we have that the previous estimate is conclusive. Let us therefore suppose that $N_3 \gg N_2$, we get :
\begin{align*}
    &(N^2)^{\varepsilon-2}(N^3)^{-2}\sum_{{n_0}}\Bigg|\sum_{\substack{|n_i|\sim N_i \ \ i=1,2,3 \\ n_2\neq n_1 \ \ n_2\neq n_3}}a_1(n_1)\overline{g_{n_2}^\omega }g_{n_3} ^\omega \gamma_{n_1,n_2,n_3}^{n_0}\Bigg|^2\leq\\
    &(N^2)^{\varepsilon-2}(N^3)^{-2}\sum_{{n_0,n_2,n_3}}\Bigg|\sum_{n_1}a_1(n_1)\gamma_{n_1,n_2,n_3}^{n_0}\Bigg|^2\le \\
    &(N^2)^{\varepsilon-2}(N^3)^{-2}\sum_{{n_0,n_2,n_3}}\Bigg(\sum_{n_1}|a_1(n_1)|^2|\gamma_{n_1,n_2,n_3}^{n_0}|^{2-\alpha}\Bigg)\Bigg(\sum_{n_1'}|\gamma_{n_1',n_2,n_3}^{n_0}|^\alpha\Bigg).
\end{align*}
We set again $1/2<\alpha<1$ and see that $|\gamma_{n_1',n_2,n_3}^{n_0}|\leeq \delta(\lambda_{n_1}-\lambda_{n_2}+\lambda_{n_3}-\lambda_{n_0}=\rho)\la n_1-n_2+n_3-n_0 \ra^{-4\alpha}(N_2)^{\alpha\rho} $ and then adding over the $n_1$ is like adding over a circle in $\Z^2$, thus summable and constant bounded. As for the first term we can bound by:
\begin{equation*}
(N_3)^{\varepsilon-2}(N_2)^{-2+\rho}\sum_{k,n_1}|a_1(n_1)|^2\#S_{\rho, n_1}^{n_2,n_3}(k)\la k\ra^{4-2\alpha}\leeq (N_2)^\varepsilon \bigg(\frac{N_2}{N_3}\bigg)^{1/2}(N_2)^{\rho}\leeq N_3^{-1/100}.
\end{equation*}
\subsubsection{Cases G and G'} We want to use the previous estimates for terms D and D', which we can use if we suppose that $N_2 \sim N_1$ (in the sense that $N_2 >N_1^{8/10}$) for D and $N_3 \sim N_1$ for D'. For the other case we get 
\begin{align*}
    &(N_1)^{2s-2+\varepsilon}(N_2)^{-2}\sum_{{n_0}}\Bigg|\sum_{\substack{|n_i|\sim N_i \ \ i=1,2,3 \\ n_2\neq n_1 }} g_{n_1} ^\omega \overline{g_{n_2}^\omega }a_3(n_3)\gamma_{n_1,n_2,n_3}^{n_0}\Bigg|^2\leq\\
    &(N_1)^{2s-2+\varepsilon}(N_2)^{-2+\rho}\sum_{n_0}\bigg( \sum_{n_1,n_2,n_3} \delta(\lambda_{n_1}-\lambda_{n_2}+\lambda_{n_3}-\lambda_{n_0}=\rho \bigg)\sum_{n_1,n_2,n_3}|a_3(n_3)|^2|\gamma_{n_1,n_2,n_3}^{n_0}|^2\le \\
    &(N_1)^{2s-2+\varepsilon}(N_2)^{-1+\rho}N_3\sum_{\substack{n_0,n_1\\n_2,n_3}}|a_3(n_3)|^2|\gamma_{n_1,n_2,n_3}^{n_0}|^2\leeq (N_1)^{2s-1+\varepsilon+\rho}(N_2)^{\varepsilon}(N_3)^{1-2s}.
\end{align*}
having first used lemma \ref{lemma:lema_contar_1} and then lemma's \ref{lemma:lema_contar_2} $i)$, with a change of $n_3\leftrightarrow n_1$, This for $(N_1)^{8/10} > N_3$ is conclusive.
\subsubsection{Cases F and F'} Once again we are to see the gaussians as a random operator acting on the deterministic term. We first observe what happens when $N_1 \gg N_2 \gg N_3$ in the same sense as before.
\begin{align*}
    &(N_1)^{2s-2+\varepsilon}(N_3)^{-2}\sum_{{n_0}}\Bigg|\sum_{\substack{|n_i|\sim N_i \ \ i=1,2,3 }} g_{n_1} ^\omega a_2(n_2)g_{n_3}^\omega \gamma_{n_1,n_2,n_3}^{n_0}\Bigg|^2\leq\\
    &(N_1)^{2s-2+\varepsilon}(N_2)^{-2}\sum_{n_0}\bigg( \sum_{n_1,n_2,n_3} \delta(\lambda_{n_1}-\lambda_{n_2}+\lambda_{n_3}-\lambda_{n_0}=\rho \bigg)\sum_{n_1,n_2,n_3}|a_2(n_2)|^2|\gamma_{n_1,n_2,n_3}^{n_0}|^2\le \\
    &(N_1)^{2s-2+\varepsilon}(N_2)^{-1}(N_3)^{1+\rho}\sum_{\substack{n_0,n_1\\n_2,n_3}}|a_2(n_2)|^2|\gamma_{n_1,n_2,n_3}^{n_0}|^2\leeq (N_1)^{2s-1+\varepsilon}(N_2)^{1-2s}(N_3)^{\rho+\varepsilon}.
\end{align*}
Therefore we have conclusive result if $N_2 \le (N_1)^{8/10}$, on the other hand repeating the same process as in case D, but with an extra $(N_3)^s$ factor we get a conclusive result if $N_3 <N_1^{1/4+\rho+\varepsilon}$, thus we may assume that $N_2 \geq (N_1)^{8/10}$ and $N_3 \geq (N_1)^{1/5}$. 

Now we may repeat the same process as for case D. We define the random matrix and estimate the norm granted by the diagonal by $(N_3)^{-2+\rho}\leq (N_1)^{-1/5}$, as for the a second term of the matrix norm, its contribution may be bounded by
\begin{equation*}
\sum_{n_0\neq n_0'}\Big|\sum_{\substack{n_1,n_1',n_2\\ n_3,n_3'}}{g_{n_1}^\omega }g_{n_3}^\omega \overline{g_{n_1'}^\omega g_{n_2'}^\omega }\gamma_{n_1,n_2,n_3}^{n_0}\overline{\gamma_{n_1',n_2,n_3'}^{n_0'}}\Big|^2.
\end{equation*}
This may be bounded by distinguishing between the case where they are different two  by two, where the $n_1=n_1';n_3\neq n_3'$ and when $n_1=n_3';n_3\neq n_1'$, as the other combinations are included as the square of the gaussians and thus of mean $0$ and independent. The first case will grant $(N_3)^{-2+\rho}\leq (N_1)^{-1/5}$, the same holds for the second, the third yields  $(N_1)^{s+\varepsilon-1/12}(N_2)^{-s+\rho}\leeq (N_1)^{-1/100}$ .
\subsection{Bilinear expression}
Now we develop (\ref{eq:terminos_decomp_2}):
\begin{equation*}
\sum_{n_2\neq n_1} (\gamma_{n_1}+\beta_{n_1})|\gamma_{n_2}+\beta_{n_2}|^2e_{n_1}(|e_{n_2}|^2-1)
\end{equation*}
and decompose in trilinear terms:
\begin{equation*}
\sum_{n_2\neq n_1} u_{n_1}\overline{u_{n_2}}v_{n_2}e_{n_1}(|e_{n_2}|^2-1).
\end{equation*}
As before we divide into cases.
\begin{center}
\begin{tabular}{cc}
\begin{tabular}{ |p{1cm}||p{1cm}|p{1cm}|p{1cm}|  }
\hline
\multicolumn{4}{|c|}{Case $N_1>N_2$} \\
\hline
& $u_1$ & $u_2$ & $v_2$ \\
\hline
I&$\gamma$&$\overline{\gamma}$&$\gamma$\\
J&$\beta$&$\overline{\gamma}$&$\gamma$\\
K&$\gamma$&$\overline{\beta}$&$\beta$\\
L&$\beta$&$\overline{\beta}$&$\beta$\\
M&$\gamma$&$Re(\gamma)$&$Re(\gamma)$\\
N&$\beta$&$Re(\gamma)$&$Re(\beta)$\\
\hline
\end{tabular}
\quad
\begin{tabular}{ |p{1cm}||p{1cm}|p{1cm}|p{1cm}|  }
\hline
\multicolumn{4}{|c|}{Case $N_2\geq N_1$} \\
\hline
& $u_1$ & $u_2$ & $v_3$ \\
\hline
I&$\gamma$&$\overline{\gamma}$&$\gamma$\\
J&$\beta$&$\overline{\gamma}$&$\gamma$\\
K&$\gamma$&$\overline{\beta}$&$\beta$\\
L&$\beta$&$\overline{\beta}$&$\beta$\\
M&$\gamma$&$Re(\gamma)$&$Re(\gamma)$\\
N&$\beta$&$Re(\gamma)$&$Re(\beta)$\\
\hline
\end{tabular}
\end{tabular} 
\end{center}
These terms are first treated here and thus aren't an adaptation in the same way as the trilinear ones are. Cases I,I',J',M,M' are treated exactly the same as cases A,B,A' and C' through Strichartz. Afterwards following the same manipulations as in section \ref{sect:term_standarization} we are left to bound:
\begin{equation}\label{eq:norma2'}
\tau^\alpha(N^2)^\varepsilon\Bigg(\sum_{{|n_0|\leq N^1}}\Bigg|\sum_{\substack{n_i\sim N_i \\ n_2\neq n_2}}u_{n_1}\overline{u_{n_2}}v_{n_2}\theta_{n_1,n_2}^{n_0}\Bigg|^2\Bigg)^{1/2}
\end{equation}
for the cases K,N' and 
\begin{equation}\label{eq:norma3'}
\tau^\alpha(N^1)^{s+\varepsilon}\Bigg(\sum_{{|n_0|\leq N^1}}\Bigg|\sum_{\substack{n_i\sim N_i \\ n_2\neq n_2}}u_{n_1}\overline{u_{n_2}}v_{n_2}\theta_{n_1,n_2}^{n_0}\Bigg|^2\Bigg)^{1/2}
\end{equation}
for the cases J,K',L,L' and N. Where for some $\lambda_0$ 
\begin{equation*}
\theta_{n_1,n_2}^{n_0}  = \Pi_{\lambda_0} \big(\la e^{-it(-\Delta+V^r)} e_{n_1}(|e_{n_2}|^2-1)e^{i(\rho+\lambda_{n_1}) t}, e_{n_0}\ra\big).
\end{equation*}
Develop in Fourier to get
\begin{equation*}
    \theta_{n_1,n_2}^{n_0}=\sum_{\substack{k_1,k_2,k_2'\\ k_2\neq k_2'}} \chi_{n_1}^{k_1}\overline{\chi_{n_2}^{k_2}}\chi_{n_2}^{k_2'}\overline{\chi_{n_0}^{k_1-k_2+k_2'}}\delta(\lambda_{n_1}-\lambda_{n_0}=\rho)
\end{equation*}
it follows that 
\begin{equation}\label{eq:theta}
    | \theta_{n_1,n_2}^{n_1+k}|^2\le \delta(\lambda_{n_1}-\lambda_{n_1+k}=\rho)({N^2})^\rho \la k\ra ^{-4}.
\end{equation}
Furthermore if we look at the fourier development of $\la e_{n_1}(|e_{n_2}|^2-1),e_{n_1+k}\ra$ we have that if $|k|>4(N_1)^{\rho}$ then $e_{n_1}\overline{e_{n_1+k}}$ and $|e_{n_2}|^2-1$ are orthogonal meaning that we can suppose that $|k|\leeq 4(N_1)^{\rho}$.
\subsubsection{Cases L and L'} A priori it would seem that in the case where $N_2$ is small and $N_1$ is big, the lack of resonances would disable any regularity gain, fourtunatley even if we don't have that $|e_n|^2=1$ we still have useful estimates. First look at
\begin{equation*}
(N_1)^{s+\varepsilon-1}(N_2)^{-2}\Bigg(\sum_{{|n_0|\leq N_1}}\Bigg|\sum_{\substack{n_i\sim N_i \\ n_2\neq n_1 }}g_{n_1}^\omega |g_{n_2}^\omega |^2\theta_{n_1,n_2}^{n_0}\Bigg|^2\Bigg)^{1/2}
\end{equation*}
and rewrite to make use of hypercontractivity estimates:
\begin{equation*}
(N_1)^{s+\varepsilon-1}(N_2)^{-2}\Bigg(\sum_{{|n_0|\leq N_1}}\Bigg|\sum_{\substack{n_i\sim N_i \\ n_2\neq n_1 }}\Big(g_{n_1}^\omega (|g_{n_2}^\omega |^2-1)+g_{n_1}^\omega \Big)\theta_{n_1,n_2}^{n_0}\Bigg|^2\Bigg)^{1/2}.
\end{equation*}
Taking the expected value and restricting in $\omega$:
\begin{equation*}
(N_1)^{s+\varepsilon-1}(N_2)^{-2}\Bigg(\sum_{{|n_0|\leq N_1}}\sum_{\substack{n_i\sim N_i \\ n_2\neq n_1}}|\theta_{n_1,n_2}^{n_0}|^2\Bigg)^{1/2}.
\end{equation*}
This is conclusive for the case L' as well as for case L and $(N_1)^{8/10}<N_2$, using (\ref{eq:theta}) we suppose that we are in the case $N_1 \gg N_2$. Let us examine thus
\begin{equation}\label{eq:localtag4}
\sum_{{|k|\leq 4(N_1)^\rho}}\sum_{\substack{n_i\sim N_i \\ n_2\neq n_1}}|\theta_{n_1,n_2}^{n_1+k}|^2=\sum_{{|k|\leq 4(N_1)^\rho}}\sum_{\substack{n_i\sim N_i \\ n_2\neq n_1}}\big|\la e_{n_1}(|e_{n_2}|^2-1),e_{n_1+k}\ra\big|\delta(\lambda_{n_1}-\lambda_{n_1+k}=\rho)
\end{equation}
we divide this in two cases, first $k=0$ and thus $\rho=0$, and $k\neq 0$. First if $k\neq 0$ we get the control 
\begin{equation*}
    (\ref{eq:localtag4})\leeq\sum_{{1<|k|\le 4(N_1)^\rho}}\sum_{\substack{n_i\sim N_i \\ n_2\neq n_1}}\big|\la e_{n_1}(|e_{n_2}|^2-1),e_{n_1+k}\ra\big|^2\delta\big(\la n_1, k\ra =-\frac{-\rho+|k|^2}{2}\big)\leeq (N_1)^{1+2\rho}(N_2)^2
\end{equation*}
which comes from the control on the number of solutions in $x$ of an equation of the form $\la x,y\ra=C$ for $y\neq 0$ and $C\in \R$, and gives the final contribution $(N_1)^{s+\varepsilon-1/2+\rho}(N_2)^{-1}$. On the other hand if $k=0$ we have to control
\begin{equation}\label{eq:localtag5}
    \sum_{\substack{n_i\sim N_i \\ n_2\neq n_1}}\big|\la (|e_{n_2}|^2-1),|e_{n_1}|^2\ra\big|^2
\end{equation}
and for that we are going to realize first that $|e_{n_1}|^2=\sum_{k\in\Pp_{n_1}} \varphi_{n_1}^k e^{ik\cdot x}$ with $\Pp_{n_1}= \{k=k_1-k_2; \ \  k_1,k_2\in \Omega_{\alpha(n_1)}\}$ and $\sum_k |\varphi_{n_2}^k|^2\le C$. On the other hand $|e_{n_2}|^2-1=\sum_{k\in\Pp^*_{n_2}} \varphi_{n_2}^k e^{ik\cdot x}$ with $\Pp_{n_2}^*= \{k=k_1-k_2, \ k\neq 0; \ \  k_1,k_2\in \Omega_{\alpha(n_2)}\}$. Remark that we have that for a $k\in (\Z^2)^*$,  in order to be in $\Pp_{n_1}$, we require that $|\la k,n_1 \ra|\leeq  (N_1)^\rho$. This means that 
\begin{equation*}
    (\ref{eq:localtag5})\le (N_2)^{\rho}\sum_{\substack{n_2 \\ k\in\Pp_{n_2}^{*}}}|\varphi_{n_2}^k|^2\sum_{n_1}|\la |e_{n_1}|^2,e^{ik\cdot x}\ra|^2\le (N_2)^{\rho}\sum_{\substack{n_2 \\ k\in\Pp_{n_2}^{*}}}|\varphi_{n_2}^k|^2 \#\big\{ | \la k,n_1 \ra|\leeq  (N_1)^\rho \big\},
\end{equation*}
and knowing that $k\neq0$ we get that $\#\big\{ | \la k,n_1 \ra|\leeq  (N_1)^\rho \big\}\le (N_1)^{1+\rho}$ and thus the final contribution $(N_1)^{s+\varepsilon-1/2+\rho}(N_2)^{-1}$.
\subsubsection{Case K'} We use the (\ref{eq:norma3'}) this time: 
\begin{equation*}
(N_2)^{s-2+\varepsilon}\Bigg(\sum_{{|n_0|\leq N_1}}\Bigg|\sum_{\substack{n_i\sim N_i \\ n_2\neq n_1 }}a(n_1)|g_{n_2}^\omega |^2\theta_{n_1,n_2}^{n_0}\Bigg|^2\Bigg)^{1/2}.
\end{equation*}
Restriction in $\omega$ grants us 
\begin{equation*}
(N_2)^{s-2+\varepsilon}\Bigg(\sum_{\substack{|n_0|\leq N^1\\n_2\sim N_2}}\Bigg|\sum_{\substack{n_1\sim N_1 \\ n_2\neq n_1 }}a_1(n_1)\theta_{n_1,n_2}^{n_0}\Bigg|^2\Bigg)^{1/2}.
\end{equation*}
Let us revisit a method used in case B' and taking some $1/2<\alpha<1$ use Cauchy-Schwartz:
\begin{equation*}
(N_2)^{s-2+\varepsilon}\Bigg(\sum_{\substack{|n_0|\leq N^1\\n_2\sim N_2}}\Big(\sum_{\substack{n_1\sim N_1 \\ n_2\neq n_1 }}|a_1(n_1)|^2|\theta_{n_1,n_2}^{n_0}|^{2-\alpha}\Big)\Big(\sum_{\substack{n_1\sim N_1 \\ n_2\neq n_1 }}|\theta_{n_1,n_2}^{n_0}|^\alpha\Big)\Bigg)^{1/2},
\end{equation*}
the second factor is bounded by a constant as we are adding for a fixed $n_0,n_2$ over the $n_1$ in a circle. On the other hand the first summand is bounded by 
\begin{equation}\label{eq:casek_prima_final}
    (N_2)^{s-2+\varepsilon}\Bigg(\sum_{\substack{n_1\sim N_1 \\ n_2\neq n_1 }}|a_1(n_1)|^2\sum_{\substack{|n_0|\leq N^1\\n_2\sim N_2}}|\theta_{n_1,n_2}^{n_0}|^{2-\alpha}\Bigg)^{1/2}\leeq (N_2)^{s-2+\varepsilon}\Bigg(\sum_{\substack{n_1\sim N_1 \\ n_2\neq n_1 }}|a_1(n_1)|^2(N_2)^2\Bigg)^{1/2}
\end{equation}
and then finally the control $(N_2)^{s-1+\varepsilon}$.
\subsubsection{Case K}
Doing the same process as for case K', this time beginning with (\ref{eq:norma2'}), we get that (\ref{eq:casek_prima_final}) becomes 
\begin{equation*}
    (N_2)^{-2+\varepsilon}\Bigg(\sum_{\substack{n_1\sim N_1 \\ n_2\neq n_1 }}|a_1(n_1)|^2\sum_{\substack{|n_0|\leq N^1\\n_2\sim N_2}}|\theta_{n_1,n_2}^{n_0}|^{2-\alpha}\Bigg)^{1/2}\leeq (N_2)^{-1+\varepsilon}\Bigg(\sum_{\substack{n_1\sim N_1 \\ n_2\neq n_1 }}|a_1(n_1)|^2\Bigg)^{1/2}\le (N_2)^{-1+\varepsilon}.
\end{equation*}
\subsubsection{Case J} We look at (\ref{eq:norma3'}) and we take the expectation of the square to get:
\begin{equation*}
    \sum_{n_0,n_0} \Big|\sum_{n_2}a_2(n_2)\theta_{n_1,n_2}^{n_0}\Big|^2\sum_{n_1,n_2,k} |\theta_{n_1,n_2}^{n_1+k}|^2|a_2(n_2)|^2 \bigg(\sum_{n_2}|a_2(n_2)|^2\bigg)\le (N_2)^{-2s}(N_1)^{2s-1+\varepsilon},
\end{equation*}
the last inequality coming from cases L and L'.
\subsubsection{Case N } We remark first that $Re(g^\omega )$ is gaussian and that if $n_1\neq n_2$ that $Re(g_{n_1}^\omega ),Re(g_{n_2}^\omega )$ are still independent. For the case N we are to bound:
\begin{equation*}
(N_1)^{s+\varepsilon-1}(N_2)^{-1}\Bigg(\sum_{{|n_0|\leq N^1}}\Bigg|\sum_{\substack{n_i\sim N_i \\ n_2\neq n_1 }}g_{n_1}^\omega Re(g_{n_2}^\omega a_2(n_2))\theta_{n_1,n_2}^{n_0}\Bigg|^2\Bigg)^{1/2}.
\end{equation*}
Take the expectation and using that $n_1\neq n_2$ and bound by:
\begin{equation*}
(N_1)^{s+\varepsilon-1}(N_2)^{-1}\Bigg(\sum_{{|n_0|\leq N^1}}\sum_{\substack{n_i\sim N_i \\ n_2\neq n_1 }} |a_2(n_2)|^2\big|\theta_{n_1,n_2}^{n_0}\big|^2\Bigg|^2\Bigg)^{1/2}\le (N_1)^{s+\varepsilon-1/2+\rho}(N_2)^{-1-s}
\end{equation*}
the last inequality being achieved through the control on the sum of the $\theta$ as seen in cases L and L'.
\subsubsection{Case N'}
We have that the same estimates suffice as for N, just that instead of applying (\ref{eq:norma3'})  we apply (\ref{eq:norma2'}).
\subsection{Approximation}
We have shown that there is a a set $\Omega$, such that the $\Omega^c$ has measure smaller than $e^{-C/\tau^\alpha}$, and such that, for $\omega \in \Omega$, there is a $u\in (e^{it(-\Delta+V)}\Gg_{-\Delta+V^r}(\omega)+B_{X_s,b}) $ satisfies:
\begin{equation*}
\left\|\int_{0}^t S(t-s)\Bigg(u|u|^2-2u\int|u|^2\Bigg)ds\right\|_{X_{s,b}}\leeq \tau^\alpha.
\end{equation*}
Meaning that for $0<\tau$ small enough $T$ is a contraction on  $u\in (e^{it(-\Delta+V^r)}\Gg_{-\Delta+V}(\omega)+B_{X_s,b})$. This grants us wellpossedness of (\ref{eq:Non_Linear_altered_equation}). We are left with showing that the solution may be approximated by the solutions to the truncated equation. We have that for $\omega\in\Omega$ and $u\in (e^{it(-\Delta+V)}\Gg_{-\Delta+V^r}(\omega)+B_{X_s,b})$:
\[\left\|Tu-Tv\right\|_{X_{s,b}}\leeq \tau^\alpha\left\|u-v\right\|_{X_{s,b}}.\]
Through the Banach fixed point theorem we have a unique solution for all $N$ of
\begin{equation*}
\begin{cases}
\Big(i\partial_t +\Delta  +V\Big)u^N =P_N(-u^N|u^N|^2+2a_N^ru^N)\\ 
u(0,x)=P_N(\Gg_{-\Delta+V^r}(\omega))(x),
\end{cases} 
\end{equation*}
in $e^{it(-\Delta+V)}P_N(\Gg_{-\Delta+V^r}(\omega))+B_{X_s,b}$. Let us show that there is a certain stability with respect to the initial conditions that will later allow us to extend the solution. Let $d>0$ and $\omega$ such that $T$ is a contraction on $e^{it(-\Delta+V)}\Gg_{-\Delta+V^r}(\omega)+B_{X_s,b}$ and take \footnote{Such a $\psi$ exists by density of $H^s$ in $H^{-\varepsilon}$ for all $\varepsilon>0$.} a $\psi\in H^s(\T^2)$ with $\|\Gg_{-\Delta+V^r}(\omega)-\psi\|_{H^s}\leq d$ we can consider the transform $T'v:=Tv+e^{it(-\Delta+V)}(\psi-\Gg_{-\Delta+V^r}(\omega))$. We have that for $v\in e^{it(-\Delta+V)}\Gg_{-\Delta+V^r}(\omega)+B_{X_s,b}$
\begin{align}\label{eq:Tprima}
&\left\|T'v-e^{it(-\Delta+V)}\Gg_{-\Delta+V^r}(\omega)\right\|_{X_{s,b}}\leq \left\|Tv-e^{it(-\Delta+V)}\Gg_{-\Delta+V^r}(\omega)\right\|_{X_{s,b}}+ \nonumber\\ &\left\|e^{it(-\Delta+V)}\psi-e^{it(-\Delta+V)}\Gg_{-\Delta+V^r}(\omega)\right\|_{X_{s,b}}\leeq (1+d)\tau^{\alpha},
\end{align}
where we use the stability of the Schrödinger semigroup. So for $\tau$ small enough $T'$ maps the space $e^{it(-\Delta+V)}\Gg_{-\Delta+V^r}(\omega)+B_{X_s,b}$ to itself. Observing that $T-T'=e^{it(-\Delta+V)}(\psi-\Gg_{-\Delta+V^r}(\omega))$, we have that $T'$ is a contraction on $e^{it(-\Delta+V)}\Gg_{-\Delta+V^r}(\omega)+B_{X_{s,b}}$. So (\ref{eq:Non_Linear_altered_equation}) also admits a solution for initial conditions $v(0)=\psi\in H^s+\Gg_{-\Delta+V^r}(\omega)$. We have that letting $u$ be the fixed point for $T$:
\begin{equation*}
\left\|v-u\right\|_{X_{s,b}}=\left\|T'v-Tu\right\|_{X_{s,b}}\leq \left\|Tv-Tu\right\|_{X_{s,b}}+\left\|e^{it(-\Delta+V)}\psi-e^{it(-\Delta+V)}\Gg_{-\Delta+V^r}(\omega)\right\|_{X_{s,b}}.
\end{equation*}
Finally, for $\tau<1$ and $t\in I_\tau$ we get:
\begin{equation}\label{eq:stab}
\left\|v(t)-u(t)\right\|_{H^s}\leq\left\|v-u\right\|_{X_{s,b}}\leq 2\left\|\Gg_{-\Delta+V^r}(\omega)-\psi\right\|_{H^s}.
\end{equation}
Now we show that we can approximate the solution of (\ref{eq:Non_Linear_altered_equation}) by the solution of the truncated version of the Cauchy problem. \begin{lem}\label{lemma:aproximation}
Let $\omega$ such that $T$ is a contraction on $e^{it(-\Delta+V)}\Gg_{-\Delta+V^r}(\omega)+B_{X_s,b}$. Let $u^N$ be the solution to the finite dimensional ODE:
\begin{equation}
\begin{cases}\label{eq_truncada}
i\partial_t u^N -\Delta u^N+ Vu^N=P_N(-u^N|u^N|^2+2a_Nu^N)\\ 
u(0,x)=P_N(\Gg_{-\Delta+V^r}(\omega))(x).
\end{cases} 
\end{equation}
Let $u$ the solution to the non truncated equation. We have that 
\begin{equation}\label{eq:aprox_troca}
\left\|(u^{N}-e^{it\big(-\Delta+V)}P_N(\Gg_{-\Delta+V^r}(\omega)\big)-(u-e^{it(-\Delta+V)}\Gg_{-\Delta+V^r}(\omega))\right\|_{X_{s,b}}\leq N^{-\alpha}.
\end{equation}
\end{lem}
\begin{proof}
Begin by fixing $0<s'<s$ and looking at the difference 
\[\left\|\big(u^{N}-e^{it(-\Delta+V)}P_N(\Gg_{-\Delta+V^r}(\omega)\big)-(u-e^{it(-\Delta+V)}\Gg_{-\Delta+V^r}(\omega))\right\|_{X_{s',b}}.\]
Using the fact that $u^{N}$ and $u$ are fixed points for $T^N$ (the flow associated to (\ref{eq_truncada})) respectively $T$ we can write $\big(u^{N}-e^{it(-\Delta+V)}P_N(\Gg_{-\Delta+V^r}(\omega)\big)-(u-e^{it(-\Delta+V)}\Gg_{-\Delta+V^r}(\omega))$ as:
\[-i \int_{0}^t e^{i(t-s)(-\Delta+V)}\Bigg(\Bigg(u|u|^2-2u\int|u|^2\Bigg)-P_N\Bigg(u^N|u^N|^2-2u^N\int|u^N|^2\Bigg)\Bigg)ds.\]
We look at 
\[\Bigg(u|u|^2-2u\int|u|^2\Bigg)-P_N\Bigg(u^N|u^N|^2-2u^N\int|u^N|^2\Bigg)\]
and once again the cubic terms can be written in terms of products of three  terms as: \[w_1(n_1)\overline{w_2}(n_2)w_3(n_3)e_{n_1}^r\overline{e_{n_2}^r}e_{n_3}^r.\] We can regroup these trilinear expansions in dyadic forms as $w_1\overline{w_2}w_3$ with one of these $w_i$ satisfying either $P_{N/3}w_i=0$ or $w_i= P_{N\le|n|\le 2N}(u-u^N)$. We then look at the work done in the Trilinear expression section. When it comes to the case $P_{N/3}w_i=0$ we have two options. First $w_i=\beta$, in this case for a $N$ big enough we get an extra $N^{-\alpha'}$ factor for some $\alpha'>0$, as we begin adding from a frequency of order $N$ instead of $1$. On the other hand if $w_i=\gamma$ we have a gain in the form of $N^{s'-s}$. For the terms in the form of $w_i=u-u^N$ we can rewrite them as the sum of 
\begin{align*}
    &e^{it(-\Delta+V)}(\Gg_{-\Delta+V^r}(\omega)-P_N(\Gg_{-\Delta+V^r}(\omega)))+\\
    &(u^{N}-e^{it(-\Delta+V)}P_N(\Gg_{-\Delta+V^r}(\omega)))-(u-e^{it(-\Delta+V)}\Gg_{-\Delta+V^r}(\omega)).
\end{align*}
The first term grants us a $N^{-\alpha'}$ gain, as for the second term being in $B_{X_s,b}$ implies the control $C\tau^\alpha \left\|(u^{N}-e^{it(-\Delta+V)}P_N(\Gg_{-\Delta+V^r}(\omega)))-(u-e^{it(-\Delta+V)}\phi)\right\|_{X_{s',b}}$. Therefore we have:
\begin{align*}
&\left\|(u^{N}-e^{it(-\Delta+V)}P_N(\Gg_{-\Delta+V^r}(\omega)))-(u-e^{it(-\Delta+V)}\Gg_{-\Delta+V^r}(\omega))\right\|_{X_{s',b}}\leq \\
&N^{s'-s}+N^{-\delta} + \tau^\alpha \left\|(u^{N}-e^{it(-\Delta+V)}P_N(\Gg_{-\Delta+V^r}(\omega)))-(u-e^{it(-\Delta+V)}\Gg_{-\Delta+V^r}(\omega))\right\|_{X_{s',b}}.
\end{align*}
And thus for $\tau$ small enough we have:
\begin{equation}
\left\|(u^{N}-e^{it(-\Delta+V)}P_N(\Gg_{-\Delta+V^r}(\omega)))-(u-e^{it(-\Delta+V)}\Gg_{-\Delta+V^r}(\omega))\right\|_{X_{s,b}}\leeq N^{-\alpha}.
\end{equation}
\end{proof}
We can finish the proof of proposition (\ref{prop:LWP}) thanks once again to the control of the $L_t^\infty H^s(\T^2)$ norm by the $X_{s,b}$ norm.
\end{proof}
\section{Extension and Gibbs measure invariance}\label{sect:measure_5}
Once we have proved that we can approximate the solution by the truncated solution we can pass to the extension to all time. We now present the proof of the main theorem.
\begin{proof}[Proof of Theorem \ref{thm:thm1}]
The extension in time comes from the invariance of the Gibbs measure by the flow paired with the previous local theory granted by corollary \ref{prop:LWP_laplaciana}. We now introduce the Gibbs measure and prove that it is well defined (in the sense that it is a weighted Wiener measure). But before let us show that the wick ordered non-linear part of the hamiltonian is almost surly finite.
\begin{prop}\label{prop:integrabilidadmu}
For all $\alpha >0$ , $V\in H^{2+\varepsilon}(\T^2)$, and $\lambda>C(V)$ big enough we have for some $\delta>0$ and all $N\in 2^\N$:
\begin{equation}\label{eq:control_ham_2d}
\mathbb{P_\omega}\left(-\frac{1}{2}\int_{\T^2}|\phi^N|^4+2a_N\int_{\T^2}|\phi^N|^2 - a_N^2+>\lambda\right) \leq e^{-C e^{-\alpha\sqrt{\lambda}}}. 
\end{equation}
where $\phi_N $ is either $ P_N(\Gg_{-\Delta+V^r}(\omega))$ or $P_N(\Gg_{-\Delta+V}(\omega))$.
\end{prop}
\begin{rem}
    We have seen in subsection \ref{subsection:proba} that the Gibbs measures defined by $-\Delta+V$ and $-\Delta+V^r$ are almost $L^\infty$ with respect to each other, meaning that we it suffices to show that this holds for $\phi_N = P_N(\Gg_{-\Delta+V^r}(\omega))$.
\end{rem}
\begin{proof}
First we rewrite and control naively
\begin{equation}\label{eq:recriture}
-\frac{1}{2}\int_{\T^2}|\phi^N|^4+2a_N\int_{\T^2}|\phi^N|^2 - a_N^2=-\frac{1}{2}\int_{\T^2}\left(|\phi^N|^2-2a_N\right)^2 + a_N^2\leq a_N^2 \leeq  \log(N)^2.
\end{equation}
\color{black}
We want to control for $N>N_0$:
\begin{equation*}
\mathbb{P}\Bigg(\Big|\big(-\frac{1}{2}\int_{\T^2}|\phi^N|^4+2a_N\int_{\T^2}|\phi^N|^2 - a_N^2\big) +\frac{1}{2}\int_{\T^2}|\phi^{N_0}|^4-2a_{N_0}\int_{\T^2}|\phi^{N_0}|^2 + a_{N_0}^2\Big|>\lambda \Bigg).
\end{equation*}
Which with the previous estimate and hypercontractivity estimates directly implies (\ref{eq:control_ham_2d}). Let us decompose in the hilbertian basis of eigenfunctions:
\begin{equation}\label{eq:decomp}
-\frac{1}{2}\int_{\T^2}|\phi^N|^4+2a_N\int_{\T^2}|\phi^N|^2-a_N^2.
\end{equation}
The idea is to divide our Hamiltonian into three parts and in a similar way as in \cite{burq2014remarksgibbsmeasuresnonlinear} estimate them separately. Let $\gamma(n_1,n_2,n_3,n_4):=\int_{\T^2}e_{n_1}^r\overline{e_{n_2}^r}e_{n_3}^r\overline{e_{n_4}^r}$, in the following write again $e_n\leftrightarrow e_n^r$. We have that (\ref{eq:decomp}) is equal to:
 \begin{align}
    -&\frac{1}{2}\sum_{(n_1,n_2,n_3,n_4) \in \Aa} \frac{g_{n_1}(\omega)\overline{g_{n_2}(\omega)}g_{n_3}(\omega)\overline{g_{n_4}(\omega)}}{\lambda_{n_1}^{1/2} \lambda_{n_2}^{1/2} \lambda_{n_3}^{1/2} \lambda_{n_4} ^{1/2}}\gamma(n_1,n_2,n_3,n_4) \label{term: decomp1}\\
    -&\sum_{\lambda_{n_{1,2}} ^{1/2}<N}\frac{|g_{n_1}(\omega)|^2|g_{n_2}(\omega)|^2}{\lambda_{n_1} \lambda_{n_2} }\gamma(n_1,n_1,n_2,n_2)+2a_N\sum_{\substack{\lambda_n ^{1/2}<N}}\frac{|g_n(\omega)|^2}{\lambda_n }-a_N^2\label{term: decomp2}\\
    +&\frac{1}{2}\sum_{\substack{\lambda_n ^{1/2}<N}}\frac{|g_n(\omega)|^4}{\lambda_n ^2}\label{term: decomp3}
 \end{align}     
With $\Aa :=\{(n_1,n_2,n_3,n_4)\in \Z^2 : \max{\lambda_{n_i}^{1/2}}\leq 2N, n_1\neq n_2,n_4;n_3\neq n_2,n_4\}$. We are once again to control the expectation of these terms. For the contribution of the first term, (\ref{term: decomp1}) , we have to bound:
\begin{equation}\label{eq:terminoquadr}
\sum_{\max \lambda_{n_i}^{1/2}>N_0} \frac{|\gamma(n_1,n_2,n_3,n_4)|^2}{\lambda_{n_1}\lambda_{n_2}\lambda_{n_3}\lambda_{n_4} }.
\end{equation}
As once again by independence of the gaussians:
\begin{equation*}
\E (g_{n_1}(\omega)\overline{g_{n_2}(\omega)}g_{n_3}(\omega)\overline{g_{n_4}(\omega)}g_{m_2}(\omega)\overline{g_{m_1}(\omega)}g_{m_4}(\omega)\overline{g_{m_3}(\omega)})=\mathbbm{1}_{\substack{n_i=m_i \\ i=1,2,3,4}}.
\end{equation*}
We have by a symmetry and a dyadic argument that we can control (\ref{eq:terminoquadr}) up to a combinatorial constant by ordering our frequencies:
\begin{equation*}
\sum_{\substack{N^i,N^1\geq N_0 \\ 1\leq i \leq 4}}(N^1N^2N^3N^4)^{-2}\sum_{\substack{|n_i| \sim N_i \\ 1\leq i \leq 4}}|\gamma(n_1,n_2,n_3,n_4)|^2.
\end{equation*}
Now we can take the sum on the $n_1$ as it gives us the sum of the squares of the projection on the $e_{n_1}$ spaces we get the control:
\begin{equation*}
\sum_{\substack{N^i,N^1\geq N_0 \\ 1\leq i \leq 4}}(N^1N^2N^3N^4)^{-2}\sum_{\substack{|n_i| \sim N_i \\ 2\leq i \leq 4}}\int_{\T^2}|e_{n_2}|^2|e_{n_3}|^2|e_{n_4}|^2.
\end{equation*}
We arrive at
\begin{align*}
&\sum_{\substack{N^i,N^1\geq N_0 \\ 1\leq i \leq 4}}(N^1N^2N^3N^4)^{-2}\sum_{n_4 \sim N^4}\int_{\T^2}|e_{n_4}|^2\sum_{n_2 \sim N^2}|e_{n_2}|^2\sum_{n_3\sim N^3}|e_{n_3}|^2\\ \leeq &C(\varepsilon,V)\sum_{\substack{N^i,N^1\geq N_0 \\ 1\leq i \leq 4}}(N^1)^{-2+\varepsilon}(N^2N^3)^{\varepsilon}.
\end{align*}
Giving us the control $C(\varepsilon,V)N_0^{-2+\varepsilon}$. The contribution of (\ref{term: decomp3}) is direct by the independence of gaussians, $N_0^{-2+\varepsilon}$. We now treat the most intrincate term (\ref{term: decomp2}) and its contribution to the difference, rewrite the term as:
\begin{align}
    -&\sum_{\substack{\lambda_{n_{1/2}}^{1/2}<N \\ \max \lambda_{n_i}^{1/2}>N_0}}\frac{(|g_{n_1}(\omega)|^2-1)(|g_{n_2}(\omega)|^2-1)}{\lambda_{n_1} \lambda_{n_2}}\int |e_{n_1}|^2|e_{n_2}|^2 \label{term:term2_1}\\
    +&2\sum_{\substack{N_0\leq \lambda_{n_{1/2}}^{1/2}<N }}\frac{|g_{n_1}(\omega)|^2}{\lambda_{n_1} \lambda_{n_2}}\int |e_{n_1}|^2(|e_{n_2}|^2-1)\label{term:term2_2}\\
    -&\sum_{\substack{N_0\leq \lambda_{n_{1/2}}^{1/2}<N}}\frac{1}{\lambda_{n_1} \lambda_{n_2}}\int |e_{n_1}|^2(|e_{n_2}|^2-1).\label{term:term2_3}
\end{align}
We control once again taking the $L^2(\omega)$ norm on the terms. For the term (\ref{term:term2_1}) we have that $n_1\neq n_2$ implies that the squared expectation renders the non diagonal terms null, meaning that we get the control
\begin{equation}\label{eq:terminillos}
\|(\ref{term:term2_1})\|_{L^2(\omega)}^2\leq \sum_{\substack{N_0\le\lambda_{n_{1/2}}^{1/2}<N }}\frac{\int |e_{n_1}|^2|e_{n_2}|^2}{\lambda_{n_1} ^2\lambda_{n_2}^2}+\sum_{\substack{\lambda_{n_{1/2}}^{1/2}<N \\ \max \lambda_{n_{1/2}}^{1/2} \ge N_0}}\frac{\int |e_{n_1}|^4\int |e_{n_2}|^4}{\lambda_{n_1} ^2\lambda_{n_2}^2},
\end{equation}
which is conclusive as we have that $|e_n|\leq C(V)$, and thus we get the control of $C(\varepsilon,\|V\|_{H^{2+\varepsilon}})N_0^{-\alpha}$. Now we come to the term (\ref{term:term2_2}) and rewrite it as:
\begin{align}
2&\sum_{\substack{N_0\leq \lambda_{n_{1,2}} ^{1/2}<N}}\frac{|g_{n_1}(\omega)|^2-1}{\lambda_{n_1} \lambda_{n_2}}\int |e_{n_1}|^2(|e_{n_2}|^2-1) \label{term:term2_2_1} \\
+2&\sum_{\substack{N_0\leq \lambda_{n_{1,2}} ^{1/2}<N}}\frac{1}{\lambda_{n_1} \lambda_{n_2}}\int |e_{n_1}|^2(|e_{n_2}|^2-1)\label{term:term2_2_2}.
\end{align}
The second term is going to contribute directly to  (\ref{term:term2_3}) and we can easily bound this thanks to the estimates done in cases L and L' in the previous section.
We can control (\ref{term:term2_2_1}) again by taking the $L^2(\omega)$ norm and we get the control of its square by:
\begin{equation}\label{eq:controlcontrolin}
4\sum_{\substack{N_0\leq \lambda_{n_{1}} ^{1/2}<N}}\frac{1}{\lambda_{n_1}^2 }\left|\int |e_{n_1}|^2\sum_{\substack{N_0\leq \lambda_{n_{2}}^{1/2}<N}}\frac{(|e_{n_2}|^2-1)}{\lambda_{n_2}}\right|^2.
\end{equation}
Write again $|e_{n_2}|^2-1=\sum_{k\in\Pp^*_{n_2}} \varphi_{n_2}^k e^{ik\cdot x}$ with $\Pp_{n_2}^*= \{k=k_1-k_2, \ k\neq 0; \ \  k_1,k_2\in \Omega_{\alpha(n_2)}\}$, and control $(\ref{eq:controlcontrolin})/4$
\begin{align*}
    &\sum_{n_1}\frac{1}{\lambda_{n_1}^2 }\Bigg| \sum_{n_2}\sum_{k\in \Pp^*_{n_2}}\frac{\varphi_{n_2}^k}{\lambda_{n_2}} \int |e_{n_1}|^2e^{-ik\cdot x}\Bigg|^2\le \sum_{n_1}\frac{1}{\lambda_{n_1}^2 }\Bigg( \sum_{n_2}\frac{|n_2|^{\rho}}{\lambda_{n_2}}\sup_{k\in \Pp^*_{n_2}} \Big| \varphi_{n_2}^k\int |e_{n_1}|^2e^{-ik\cdot x}\Big|\Bigg)^2\\
    &\le \sum_{\substack{n_1,n_2, \\n_2'}}\frac{1}{\lambda_{n_1}^2 } \frac{|n_2|^{2\rho}}{\lambda_{n_2}^{1-\varepsilon}\lambda_{n_2'}^{1+\varepsilon}}\sup_{k\in \Pp^*_{n_2}} \Big|\int |e_{n_1}|^2e^{-ik\cdot x}\Big|^2\sup_{k\in \Pp^*_{n_2'}} | \varphi_{n_2'}^k|^2\\ 
    &\le \sum_{n_1,n_2}\frac{1}{\lambda_{n_1}^2 } \frac{|n_2|^{2\rho}}{\lambda_{n_2}^{1-\varepsilon}}\sup_{k\in \Pp^*_{n_2}} \Big|\int |e_{n_1}|^2e^{-ik\cdot x}\Big|^2.
\end{align*}
Controlling the sum over the $n_2$ as the $L^2$ norm of $|e_{n_1}|^2$ we get
\begin{equation*}
    \sum_{n_1}\frac{1}{\lambda_{n_1}^2 } \sum_{n_2} \frac{|n_2|^{2\rho}}{\lambda_{n_2}^{1-\varepsilon}}\sup_{k\in \Pp^*_{n_2}} \Big|\int |e_{n_1}|^2e^{-ik\cdot x}\Big|^2\leeq \sum_{n_1}\frac{1}{\lambda_{n_1}^2 }\le (N_0)^{-2^{-}}.
\end{equation*}
Therefore we get the control $(\ref{eq:controlcontrolin})\leq C(V)N_0^{-\delta}$. 
Finally we arrive by using the previous (\ref{eq:recriture}) and taking $N_0$ such that $\lambda > 2\log (N_0)^2$ the control $e^{-C(\varepsilon,\|V\|_{H^{2+\varepsilon}}) e^{-\delta\sqrt{\lambda}}}$.
\end{proof}
This means that we can define the Gibbs measure\footnote{The last expression is formal}:
\begin{equation}\label{eq:measure_winer_no_truncada}
\mu:= \exp\left(-\frac{1}{2}\int_{\T^2}:|\phi|^4:\right)\rho_{-\Delta+V}=\exp(H)\frac{dU}{Z} . 
\end{equation} 
and see it as a well defined weighted Wiener measure. We can write $\mu$ formally as $e^{-H(\phi)}du$ with $du$ being a pseudo weighted Lebesgue measure on the phase space. On the other hand define the truncated Gibbs measure that corresponds to the truncated hamiltonian:
\begin{equation}\label{eq:measure_winer}
\mu_N:= exp\left(-\frac{1}{2}\int_{\T^2}|\phi^N|^4+2a_N\int_{\T^2}|\phi^N|^2 - a_N^2\right)\rho_N=e^{-H^N(\phi^N)}du_N. 
\end{equation}
Note $\Lambda_N=\#\{n\in \Z^2,(\lambda_n^r)^{1/2} \}\sim N^2$, and $du_N$ be the Lebesgue product measure on $\mathbb{C}^{\Lambda_N}$.
We finally extend the solution for all time by applying a by now fairly standard para-theorem that uses local in time existence and invariant measure to give all time existence. Fix $\delta>0$ and $T\geq 0$ let the interval $[0,T]$ be divided in $\lceil T/\tau\rceil$ intervals of fixed length $\tau>0$ to be chosen later. Now we will show that we can extend the lemma \ref{lemma:aproximation} recurrently on the sub-intervals and at the same time extend the result granting us all time existence and convergence almost surely. Note for $M,N\in 2^\N$ $\Phi^{N}_t,\Phi^{N}_t$ the flow at time $t\in[0,\tau]$ corresponding to the equation (\ref{eq_truncada}). Take for some $M\in 2^\N$ a $\psi=\Gg_{-\Delta+V}(\omega),\omega\in\Omega$ in the set of good data we have for $N<M$ deduce from (\ref{eq:aprox_troca}):
\begin{equation}\label{eq:loquequieroextender}
\|(\Phi^M_t-\Phi_t^N(t)P_N-e^{it(-\Delta+V)}(Id-P_N))\psi^M\|_{H^s}\leq N^{-\alpha}
\end{equation}
with $\alpha > 0$ and $t\in[0,\tau]$. Writing $\UU_M(t):=\Phi_t^M\psi$, a priori on $[0,\tau]$, and eventually defined for $t\in [0,T]$ with the flow being extended after making more considerations on the choice of $\psi$. We have:
\begin{equation}\label{eq:primerpaso}
\|\UU_M(\tau)-(\Phi_t^N(\tau)P_N-e^{i\tau(-\Delta+V)}(Id-P_N))\psi^M\|_{H^s}\leq N^{-\alpha}.
\end{equation}
Now as anticipated, we may suppose that $\psi^M=\Gg^M_{-\Delta+V}(\omega), \ \ \omega \in \Omega_\delta \cap (\Phi_\tau)^{-1}\Omega_\delta$, meaning that $\UU_M(\tau)\in \Omega_\delta^{M}$. We now accumulate some inequalities that will allow us to extend and control our solution. We do this by repeating the previous step thanks to local existence and we get that for $t\in [0,\tau]$
\begin{equation*}
\|(\Phi^M_t-\Phi_t^N(t)P_N-e^{it(-\Delta+V)}(Id-P_N))\UU_M(\tau)\|_{H^s}\leeq N^{-\alpha}.
\end{equation*}
and as before rewrite:
\begin{equation*}
\|\UU_M(\tau+t)-\big(\Phi_t^N(t)P_N-e^{it(-\Delta+V)}(Id-P_N)\big)\UU_M(\tau)\|_{H^s}\leeq N^{-\alpha}.
\end{equation*}
The transform $e^{it(-\Delta+V)}$ is unitary, thus from hypothesis we can apply (\ref{eq:primerpaso}) and get
\begin{equation}
\|e^{it(-\Delta+V)}(Id-P_N)\UU_M(\tau)-e^{i(t+\tau)(-\Delta+V)}(Id-P_N)\psi^M\|_{H^s}\leeq N^{-\alpha}.
\end{equation}
Furthermore knowing that $\|P_N\|_{\Ll(L^2)}\le1$  we get:
\begin{equation}\label{eq:segundopaso}
\|P_N\UU_M(\tau)-\Phi^N_{\tau}P_N\psi\|_{H^s}\leeq N^{-\alpha}.
\end{equation}
By hypothesis $\UU_M(\tau)\in \Omega_\delta^{M}$ we can apply the stability estimation (\ref{eq:stab}) to the preceding estimate and get 
\begin{equation}\label{eq:tercerpaso}
\|\Phi^N_tP_N\UU_M(\tau)-\Phi^N_{t+ \tau}P_N\psi\|_{H^s}\leq N^{-\alpha}.
\end{equation}
We combine the previous to extend (\ref{eq:loquequieroextender}) for $t\in[0,2\tau]$:
\begin{equation}\label{eq:el_gran_paso_final}
\|\Big(\Phi^M_{t+\tau}-\Phi^N_{t+\tau}-e^{i(t+\tau)(-\Delta+V)}(Id-P_N)\Big)\psi\|_{H^s}\leq(\ref{eq:primerpaso})+(\ref{eq:segundopaso})+(\ref{eq:tercerpaso})\leq N^{-\alpha}.
\end{equation}
We now generalize for $t\in [0,T]$, let:
\begin{equation}\label{eq:omega}
\Oo^M:= \bigcap_{k\leq\lceil \frac{T}{\tau}\rceil} (\Phi^{M}_{k\tau})^{-1}(\Omega).
\end{equation}
Now by construction and the Hamiltonian's invariance we get that $\mu_M$ is invariant by the flow $\Phi^M$. We get that $\mu_M((\Oo^M)^c)\leq T/\tau e^{-1/\tau^\beta}$. This grant us that for any given $\sigma>0$ we can make a choice of $\tau$ such that there is a set $\Oo^M, \mu_M((\Oo^M)^c)\leq\sigma$ such that for initial data $\omega\in \Oo^M, \psi=\Gg_{-\Delta+V}(\omega) $:
\begin{equation}
\|(\Phi^M_{t}-\Phi^N_{t}-e^{it(-\Delta+V)}(Id-P_N))\psi\|_{H^s}\leq C(\sigma)N^{-\alpha}.
\end{equation}
Now we have the limit of measures:
\begin{lem}\label{lemma:aprox}
Let $E_N=Vect(\{e_n^r\}_{|n|\leq N})$ and let $U\subset H^s(\T^2)$ be an open set, with $s<0$, then we have:
\begin{equation*}
\rho(U)=\lim_{N\to \infty}\rho_N(E_N\cap U)
\end{equation*}
\begin{equation*}
\mu(U)=\lim_{N\to \infty}\mu_N(E_N\cap U)
\end{equation*}
\end{lem}
This follows from the hypercontractivity estimates on the gaussian measure and the comparability of $\rho_N^{-\Delta+V}$ and $\rho_N^{-\Delta+V^r}$. For a detailed proof \cite{Z92}. Now we can let $M\rightarrow \infty$ and using triangle inequality we get for $\mu(\Oo^c)\le \sigma$, $\omega\in \Oo, \psi=\Gg_{-\Delta+V}(\omega) $:
\begin{equation}\label{eq:ultima}
\|((\Phi^{N_1}_{t}-e^{it(-\Delta+V)})P_{N_1})-(\Phi^{N_2}_{t}-e^{it(-\Delta+V)})P_{N_2}))\psi\|_{H^s}\leq C(\sigma) (N_1 \wedge N_2)^{-\alpha}.
\end{equation}
This last equation gives us by completeness that $(\Phi^{N}-e^{it(-\Delta+V)})\psi$ converges in $H^s$ and thus $\Phi^{N}\psi$ converges weakly to some $\UU(t)=\Phi^\infty_t\psi$. And using again $(\ref{eq:ultima})$ we get:
\[\|\UU(t)-\Phi^{N}_{t}P_{N}\psi\|_{H^s}\leq C(\sigma)N^{-\alpha}\]
This coupled with the almost sure convergence in $\omega$ of $c_N(\omega)\rightarrow c_\infty(\omega)$ grants us the weak convergence for almost al time in $N$ of:
\[e^{2ic_\infty(\omega) t}\UU(t)-e^{2ic_N(\omega) t}\Phi^{N}_{t}P_{N}\psi_\omega\]
for $\omega\in\Omega$.
\end{proof}
\printbibliography
\end{document}